\documentclass[12pt]{article}

\usepackage{xcolor}
\usepackage{amsfonts}
\usepackage{amssymb}
\usepackage{amsmath}
\usepackage{amsthm}
\usepackage[margin=.75in]{geometry}

\usepackage{caption}
\usepackage{subcaption}
\usepackage{graphicx}

\newtheorem{theorem}{Theorem}
\newtheorem{remark}{Remark}
\newtheorem{proposition}{Proposition}
\newtheorem{lemma}{Lemma}
\newtheorem{corollary}{Corollary}
\newtheorem{hyp}{Hypothesis}

\newcommand{\dd}{\mathrm{d}}
\newcommand{\e}{\mathrm{e}}
\newcommand{\ul}{\underline}
\newcommand{\lap}{\Delta}

\newcommand{\nnR}{\mathbb{R}_{\geq0}}

\newcommand{\Jeps}{{\mathcal{J}^{(\epsilon)}}}
\newcommand{\rheps}{\rho^{(\epsilon)}}
\newcommand{\ueps}{{u^{(\epsilon)}}}
\newcommand{\eps}{\epsilon}

\newcommand{\dom}{{\mathbb{T}^d}}
\newcommand{\Hm}[1]{{H^{#1}(\dom)}}
\newcommand{\Lp}[1]{{L^{#1}(\dom)}}

\newcommand{\Td}{\mathbb{T}^d}

\newcommand{\sys}{equations \eqref{eq:pde}-\eqref{eq:tau}}
\newcommand{\sysaut}{equations \eqref{eq:pde} and \eqref{eq:uaut}-\eqref{eq:tauaut}}

\newcommand{\Aopaut}{\mathcal{A}^{(\eps)}}
\newcommand{\Aopnon}{\mathcal{A}^{(\eps)}_t}

\title{Local existence of solutions to a nonlinear autonomous PDE model for population dynamics with nonlocal transport and competition}

\author{Michael R. Lindstrom; mike.lindstrom@utrgv.edu}

\begin{document}

\maketitle

% REQUIRED
\begin{abstract}
In this paper, we prove that a particular nondegenerate, nonlinear, autonomous parabolic partial differential equation with a nonlocal mass transfer admits the local existence of classical solutions. The equation was developed to qualitatively describe temporal changes in population densities over space through accounting for location desirability and fast, long-range travel. Beginning with sufficiently regular initial conditions, through smoothing the PDE and employing energy arguments, we obtain a sequence of approximators converging to a classical solution. 
\end{abstract}

{\bf Keywords:} local existence, nonlocal operator, nonlinear diffusion, non-degenerate parabolic PDE, autonomous PDE

\section{Introduction}

\label{sec:intro}

Partial differential equations (PDEs) are a valuable means of modelling physical and biological systems. These models help in understanding key qualitative features and in making quantitative predictions. PDE models are often derived from consideration of physical principles, but it is not always clear if solutions even exist; or, if they do exist, what properties they enjoy such as smoothness, uniqueness, etc. The canonical example of this is the Navier Stokes Equations \cite{navier,bertozzi} in fluid dynamics where global existence of solutions is still unknown \cite{devlin2002millennium}. If a PDE is to be studied numerically, knowing an appropriate solution space is then helpful in selecting a suitable numerical method --- see for example books addressing numerical schemes for different classes of PDE \cite{leveque2002finite,zienkiewicz1977finite,holden2010splitting}. 

In this paper, we confirm that, under suitable hypotheses, there exist classical solutions to a particular nonlinear parabolic integro-differential equation model. The model of interest \cite{m3as} was derived to describe populations of people experiencing homelessness. The model is ecological in nature, taking into consideration the desirability of a location (such that individuals may be more likely to stay at a location when it is more desirable), nonlocal travel (over a short time, an individual can travel a large distance), and that desirability decreases as the local population in an area increases (due to reduced available resources per individual). The domain itself imposes additional data on the PDE through spatiotemporal variations in entry and exit rates. The understanding of this model is important as it was the first of its kind, to the author's knowledge, PDE model describing homelessness. While many important properties have been proven on solutions to the model \cite{m3as} assuming their existence, the existence of solutions has not been explored. This paper thus establishes that the model previously derived is a theoretically reasonable framework --- from a mathematical standpoint --- through which to study homelessness. Understanding this model mathematically, and the types of solutions that exist, can provide a solid foundation for it being used to better understand the dynamics of homelessness, for extending the model by incorporating more details, for adapting the model suitably to other applications, and even to fitting the model to real-world data. 

Nonlocal parabolic PDEs arise in many settings including mathematical finance \cite{abergel2010nonlinear}, models of lieukemia \cite{busse2017asymptotic}, and population dynamics \cite{holmes1994partial}. The use of PDEs to study populations is quite extensive. Some examples include the classical Keller-Segel model for a mobile population in the presence of a chemo-attractant \cite{horstmann20031970}, the Fisher PDE to model populations undergoing dispersion and logistic growth \cite{holmes1994partial}, Lotka-Volterra systems for competition \cite{holmes1994partial,berestycki2018predators}, models of pedestrians in crowds \cite{hughes2000flow}, biological aggregation \cite{bertozzi2010existence}, and models of crime dynamics \cite{short2010dissipation,zipkin2014cops,crime}.

Rigorous analysis has proved that Patlak-Keller-Segel models with degenerate diffusion are locally well-posed in $2$ dimensions and globally well-posed in $3$ dimensions \cite{bedrossian2011local}. The Patlak-Keller-Segel models are nonlocal in space, like the model studied here. Beyond mathematical models themselves, solutions are known to exist for quite a broad range of nonlocal parabolic PDEs \cite{alibaud2007existence,ackleh2000existence}, but our system violates some hypotheses used to establish these results, such as monotonicity in the dependent variable or derivatives of nonlinear terms appearing. For simplicity, we assume the data in our system are smooth, but 
 other authors also study PDEs with nonregular data \cite{dall1992existence}. Two somewhat distinguishing features of our model include that the highest derivative is acting on a nonlinear function of the dependent variable and the nonlocal operator in our model is not a convolution. The work of \cite{crime,greer} that establish existence of solutions to their respective models are local and have the highest derivative acting on a linear function of the dependent variable. Although our PDE of interest has a high degree of nonlinearity, methodologically, we proceed to prove local existence through reasonably similar means: establishing local existence of smoothed solutions and extracting a subsequence that tends to a valid solution. The use of Lyapunov-type methods to establish existence of solutions to PDEs is another common strategy \cite{lyap}. 

With our work placed in context, the remaining paper is organized as follows: the PDE model, the main theorems, and the key steps of their proofs are provided in Section \ref{sec:model}; important results for the proofs are given in Section \ref{sec:theory}; we prove the main results through a series of smaller steps in Section \ref{sec:steps}; lastly, in Section \ref{sec:conclusion}, we conclude our work and contemplate further directions. In Section \ref{app:nom} of the Appendix, key definitions and nomenclature are explained, and proofs of various supporting results used in this paper are given. A transcription of proofs of important mollifier results is given in Section \ref{app:greer} of the Appendix.

\section{Background}

\label{sec:model}

\subsection{Model Equation and Background}

In \cite{m3as}, a nonlinear PDE was derived to describe the mean field dynamics of agents on a lattice with a nonlocal travel term. The original equation was developed as a first attempt to use PDEs to qualitatively describe homeless population densities.
 Over a bounded, $d$-dimensional spatial domain $\Omega \subset \mathbb{R}^d$, we consider the evolution of a density field $\rho(t,x)$ at time $t \geq 0$ and position $x \in \Omega$ given by
\begin{align}
\rho_t &= \delta \lap( \rho u ) + \eta - \omega \rho + I[\gamma \rho u] - \gamma \rho u \label{eq:pde} \\
%\rho(0,x) &= \rho_0(x) \label{eq:ic} \\
u &:= (\kappa(t,x), \rho(t,x)) \mapsto u^+ - \kappa \frac{u^-}{1 + \rho/\tilde \rho} \label{eq:u} \\
I[q](t,x) &:= \int_\Omega \tau(t,y,x) q(t,y) \dd y, \quad \forall x \in \Omega, t \geq 0 \label{eq:I} \\
\int_\Omega \tau(t,y,x) \dd x &= 1, \quad \forall y \in \Omega, t \geq 0, \label{eq:tau}
\end{align}
with a prescribed initial density at $t=0$ and suitable boundary conditions on $\partial \Omega.$ Depending on scalings chosen, the system as such can be dimensional or dimensionless.

We assume that
\begin{hyp}[General Hypotheses] \label{hyp:all} $ $ 
\begin{enumerate}
\item For $x, y \in \Omega$ and $t \geq 0$: $0 \leq \kappa(t,x) \leq 1$, $\eta(t,x) \geq 0$, $\omega(t,x) \geq 0$, $\gamma(t,x) \geq 0$, and $\tau(t,y,x) \geq 0$;
\item All of $\kappa$, $\eta$, $\omega$, $\gamma$, and $\tau$ are $\mathcal{C}^\infty$-smooth functions in all time and space arguments with mixed partial derivatives of all orders also $\mathcal{C}^\infty$-smooth;
\item For each $m = 0, 1, 2, ...$, there is a constant $\Xi_m$ so that the $L^\infty(\Omega)$-norm of all $m^{\text{th}}$-order derivatives of $\kappa$, $\eta$, $\omega$, $\gamma$, and $\tau$ are bounded by $\Xi_m < \infty$ for all $t \geq 0$; and
\item $\delta, u^-, u^+$, and $\tilde \rho$ are constants such that $0 < \delta$, $0 < u^- < u^+ < 1,$ and $\tilde \rho > 0$.
\end{enumerate}
\end{hyp}

The terms in the model, in the case of homelessness and in a dimensional framework, can be explained as follows:
\begin{itemize}
\item $\eta$ is an entry rate: it is the mean number of people entering the population of people experiencing homelessness per unit area per unit time. Local factors such as increases in rent could increase $\eta$.
\item $\omega$ is an exit rate: it is the rate that a person experiencing homelessness ceases to be homeless. Local features such as more affordable housing units or job placement could increase $\omega$. 
\item $\gamma$ is a rate of long distance travel: it is the rate that a person who is homeless and wishes to relocate will travel a long distance in the city. Local features such as the number of public transit options affect $\gamma$.
\item $\tau$ models the probability (density) in moving from one region to another over a long distance at any given time. Factors such as routes of public transit or general knowledge about what is happening in different parts of the city affect $\tau$. The fact that $\int_\Omega \tau(t,y,x) \dd x = 1$ amounts to the fact that, starting from $y$ (getting on a bus, say) at time $t$, the probability of going somewhere must be $1$.
\item $I$ is a nonlocal operator that moves the density between points with a transfer kernel $\tau$: $I[\gamma \rho u](t,x)$ is the number of people experiencing homelessness moving to a given location per unit area per unit time through all long-range travel over $\Omega$. From the interpretations of $\tau$ and $\gamma$, we note that $\tau$ is likely heavily influenced by both public transit and general knowledge of how different parts of the city are changing.
\item $u$ is an ``unattractiveness" term and it represents the probability a person who is homeless wishes to leave their current location, either by walking to a nearby location or by travelling long-range in the city. It takes into account the effects of competition/crowding, since $u$ increases with $\rho$. While $\rho$ is the local population density, the probability a given individual wishes to relocate is $u$, giving an effective quantity $\rho u$ to diffuse/move.
\item $\kappa$ describes the amount of resources in a location: for larger $\kappa$ values, the unattractiveness is smaller.
\end{itemize}
The constants also have meaning as follows: 
\begin{itemize}
\item $\delta$ is a diffusivity: by assuming the local movement is roughly random and taking a diffusive limit, there is diffusion. 
\item $u^+$ is the maximum possible unattractiveness of any location: we assume $u^+<1$ because, observationally, we note that individuals experiencing homelessness may be found in all parts of a city, even areas with limited or no resources.
\item $u^+-u^-$ is the minimum possible unattractiveness of any location: we assume that $u^- < u^+$ so $u^+-u^-$ is strictly positive. Because $u$ is representing the probability a person experiencing homelessness will leave the area, we do not allow $u$ to be $0$ because that would be too deterministic and not allow a person to move about. 
\item $\tilde \rho$ controls how sensitive the unattractiveness is to the population density: as $\tilde \rho$ increases, the unattractiveness $u$ is affected less and less by changes in $\rho$.
\end{itemize} We note that the unattractiveness can never reach zero for any $\rho$ and the equation is nondegenerate.

The model does not directly account for aggregation effects whereby individuals who are homeless may choose to group together, nor does it account for how the homeless population itself may affect the features of the city. In a sense, it is a ``leading order" model to describe homelessness in PDE fashion, where additional effects could be added.

We remark that some of the terms in this model may be relevant in other settings, too. Models of foraging do take into account tendencies of species leaving a patch of resources after some time \cite{bartumeus2009optimal}, which is what $u$ effectively describes. In the case of models for locusts, the population density does affect the dispersion \cite{georgiou2020modelling}. Some species living in the ocean are heavily influenced and dispersed by ocean currents \cite{pringle2017ocean}, which may provide nonlocal-like behaviour similar to $I$ by rapidly moving the species around. A nonlocal operator $I$ could even be relevant for describing the general population where long-range movement occurs in commuting to/from work or between cities.

Various qualitative properties of solutions were proven in \cite{m3as}, assuming the existence of classical solutions, with $\Omega = \mathbb{T}^{d}$. In this paper, we show that classical solutions do exist on the torus. The torus, owing to the empty boundary, makes preliminary analysis simpler, without having to deal with complex geometries or boundary conditions. However, large cities often do have repeated pockets of residential and commercial districts, which gives a loosely periodic pattern, and thus the model could still be insightful.

For completeness, we state the uniqueness result for classical solutions already proven \cite{m3as}.
\begin{lemma}[Uniqueness of Smooth Solutions]
\label{thm:m3as_unique}
Let $T>0$ and suppose \\ $\rho_1, \rho_2 \in \mathcal{C}^0([0,T], \mathcal{C}^2(\Td)) \cap \mathcal{C}^1([0,T], \mathcal{C}^0(\Td))$ 
are two solutions to Eqs. \ref{eq:pde}-\ref{eq:tau} under Hypothesis \ref{hyp:all}. If $\rho_1$ and $\rho_2$ have identical and strictly positive $\mathcal{C}^2(\Td)$ initial conditions at $t=0$, then $\rho_1 = \rho_2$ on $0 \leq t \leq T$.
\end{lemma}

\subsection{Main Results and Proof Structure}

The notation we use throughout this manuscript is quite standard but section \ref{app:nom} of the Appendix defines much of our notation and various function spaces.

While we do obtain a preliminary result for the system \sys\ under Hypothesis \ref{hyp:all}, our focus is on the autonomous system, i.e., Eq.  \eqref{eq:pde} where
\begin{align}
u &:= (\kappa(x), \rho(t,x)) \mapsto u^+ - \kappa \frac{u^-}{1 + \rho/\tilde \rho} \label{eq:uaut} \\
I[q](t,x) &:= \int_\Omega \tau(y,x) q(t,y) \dd y, \quad \forall x \in \Omega, t \geq 0 \label{eq:Iaut} \\
\int_\Omega \tau(y,x) \dd x &= 1, \quad \forall y \in \Omega, \label{eq:tauaut}
\end{align}
with 
\begin{hyp}[Autonomous Hypotheses] \label{hyp:aut} $ $ 
\begin{enumerate}
\item For $x, y \in \Omega$: $0 \leq \kappa(x) \leq 1$, $\eta(x) \geq 0$, $\omega(x) \geq 0$, $\gamma(x) \geq 0$, and $\tau(y,x) \geq 0$;
\item All of $\kappa$, $\eta$, $\omega$, $\gamma$, and $\tau$ are $\mathcal{C}^\infty$-smooth functions in their arguments with mixed partial derivatives of all orders also $\mathcal{C}^\infty$-smooth;
\item For each $m = 0, 1, 2, ...$, there is a constant $\Xi_m$ so that the $L^\infty(\Omega)$-norm of all $m^{\text{th}}$-order derivatives of $\kappa$, $\eta$, $\omega$, $\gamma$, and $\tau$ are bounded by $\Xi_m < \infty$; and
\item $\delta, u^-, u^+$, and $\tilde \rho$ are constants such that $0 < \delta$, $0 < u^- < u^+ < 1,$ and $\tilde \rho > 0$.
\end{enumerate}
\end{hyp}
In other words, the data are time-independent. Our main result will be Theorem \ref{thm:local_exist} below:

\begin{theorem}[Local Existence] \label{thm:local_exist}
Let $d \in \mathbb{N}$. Let $\rho_0 \in \Hm{m}$ with $\rho_0 > 0$ a.e. and $m \in \mathbb{Z}$ with $m > d^* :=  3+d.$ Then there exists a $T_m>0$ such that $\rho^*(t,x)$ is the unique $\mathcal{C}^0([0,T_m]; \Hm{m}) \cap \mathcal{C}^1([0,T_m]; \Hm{m-2})$ solution to \sysaut\ with $\rho^*(0,\cdot) = \rho_0.$ Moreover, $\rho^*$ is a classical solution, namely $\rho^* \in \mathcal{C}^0( [0,T_m]; \mathcal{C}^2(\dom)) \cap \mathcal{C}^1( [0,T_m]; \mathcal{C}^0(\dom)).$
\end{theorem}

Establishing these results follows the steps below. See also Table \ref{tab:exec} for a more detailed summary.

\begin{enumerate}
\item We prove that regularized versions of \sysaut, smoothed by a parameter $\epsilon > 0$, have solutions $\rho^{(\epsilon)}$ that exist locally in time for an interval that could depend on $\epsilon$.
\item For large enough $m$ we then prove that the energy $\frac{1}{2} ||\rheps||_{\Hm{m}}^2$ is bounded by a function $E_m(t)$, which exists independently of $\rheps$ on some interval $[0,\varpi_m]$ where $\varpi_m$ does not depend upon $\epsilon.$ Additionally, we show that an interval $[0,T_m]$, with $T_m>0$ independent of $\eps$, exists such that all solutions $\rheps \in \mathcal{C}^1([0,T_m]; \Hm{m}).$
\item From the bounded energy, the Aubin-Lions-Dubinski{\u\i} lemma guarantees there is a subsequence converging strongly in $\mathcal{C}^0( [0,T_m]; \Hm{m'})$ for $m' < m$, $m$ and $m'$ large enough. We also establish a subsequence that converges weakly to an element of $L^\infty([0,T_m]; H^m(\dom)).$
\item From density and continuity arguments, we find a subsequence converging to some $\rho^* \in \mathcal{C}^0([0,T_m]; \Hm{m})$. 
\item Through convergence analysis, we show the limit does in fact solve \sysaut\ with appropriate initial conditions and belongs to suitable solution spaces.
\end{enumerate}

\begin{table}[h]
\centering
\begin{tabular}{l | l | p{4.5in} }
\bf Step & \bf Result & \bf Purpose \\
\hline
1 & Proposition \ref{prop:bounda} & establishes time derivative operator for regularized solutions $\rheps$ parameterized by $\eps$ is locally Lipschitz from $H^m(\dom) \to H^m(\dom)$ \\
\hline
& Theorem \ref{thm:localreg} & proves regularized solutions exist locally in time for initial data $\rho_0 \in \Hm{m}$ with $m>d/2$, where that time could depend on $\eps$ and $m$ \\
\hline
\hline
2 & Proposition \ref{prop:Ebound} & shows that for all $\eps>0$, the energy $\frac{1}{2} ||\rheps||_{\Hm{m}}^2$ is uniformly bounded on a time interval \\
\hline
& Proposition \ref{prop:red} & shows all regularized solutions exist up to a time $T_m$ independent of $\eps$ \\
\hline
\hline
3 & Lemma \ref{lem:weakConverge} & shows the sequence $\rheps$ has a subsequence converging weakly in $L^2([0,T_m]; \Hm{m})$ to a function $\rho^*$ in $L^\infty([0,T_m]; \Hm{m})$ \\
\hline
& Lemma \ref{lem:strongConverge} & shows the sequence $\rheps$ also has a subsequence converging strongly in $L^\infty([0,T_m]; \Hm{m'})$ for all $m'<m$ \\
\hline
\hline
4 & Proposition \ref{prop:HmHm} & shows the limit of the subsequence of $\rheps$, $\rho^* \in L^\infty([0,T_m]; \Hm{m})$ \\
\hline 
\hline
5 & Proposition \ref{prop:uniformC} & shows the sequence of derivatives $\rheps_t$ converges uniformly (needed to exchange limit and derivative) \\
\hline
& Theorem \ref{thm:local_exist} & gives the main result, that for $m>d+3$, classical solutions exist locally in time and they are in $\mathcal{C}^0([0,T_m]; \Hm{m}) \cap \mathcal{C}^1([0,T_m]; \Hm{m-2})$
\\
\hline
\end{tabular}
\caption{Summary of key stages in the proof of Theorem \ref{thm:local_exist}.}
\label{tab:exec}
\end{table}

\begin{remark}
While we establish local existence of classical solutions, proving global existence of solutions is beyond the scope of our work. Proving global existence would likely require first establishing global existence to the regularized systems and then showing this global existence is preserved in the limit. Through our work, we only establish local existence of the regularized solutions, with an interval of existence that may depend upon the initial condition (and solution regularity imposed) and we don't identify a natural way to ``glue" local regularized solutions together to establish global regularized solutions.
\end{remark}

\section{Theoretical Background}

\label{sec:theory}

This section provides some key definitions of tools (such as mollifiers) that are used in the proofs; states key results from known literature; and, for the case of our particular model system, various minor lemmas and computations are presented. Additional results relevant in establishing these results and the definitions of different solution spaces, etc., can be found in Appendix \ref{app:nom}. A reader more focused on the main steps in proving Theorem \ref{thm:local_exist} can move to Section \ref{sec:steps} but may wish to refer back to this section to fill in some details.

\subsection{Key Tools}

The two theorems below \cite{bertozzi} are useful for establishing existence of regularized solutions and in extending those solutions.

\begin{lemma}[Picard-Lindeloff for Autonomous ODE in Banach Space]
	\label{thm:picardb}	
	Let $O \subset B$ be an open subset of a Banach space $B$ and let $f: O \to B$ be a mapping that is locally Lipschitz, i.e., for any $v_0 \in O$, there is $L>0$ and a neighborhood $U_0$ of $v_0$ so that $||f(v_1)-f(v_2)||_X \leq L||v_1-v_2||_X$ for all $v_1, v_2 \in U_0.$ Then for any $v_0 \in O$, there is a time $T>0$ so that the initial value problem $v'(t) = f(v)$ with $v(0) = v_0$ has a unique solution $v \in \mathcal{C}^1([0,T); O).$
\end{lemma}

\begin{lemma}[Autonomous Extension]
	\label{thm:extb}	
	Let $O \subset B$ be an open subset of a Banach space $B$ and let $f: O \to B$ be a locally Lipschitz operator. Let $v \in \mathcal{C}^1([0,T); O)$ be the unique solution to the initial value problem $v'(t) = f(v)$ with $v(0) = v_0$ found by Lemma \ref{thm:picardb}. Then either the solution $v(t)$ exists globally in time or $T<\infty$ and $v$ leaves the open set $O$ as $t \uparrow T.$
\end{lemma}

The Aubin-Lions-Dubinski{\u\i} Lemma is used in extracting a strongly convergent subsequence of approximators \cite{aubin}.

\begin{lemma}[Aubin-Lions-Dubinski{\u\i}]
	\label{thm:aubin}
	Let $X$, $B$, and $Y$ be Banach spaces with $X \subset \subset B$ (compact embedding) and $B \hookrightarrow Y$ (continuous embedding). For $T>0$, let $U \subset L^p( [0,T] ; X )$ and let $V = \{ \dot u | u \in U \}$ be bounded in $L^p( [0,T]; X)$ and $L^q ( [0,T]; Y)$, respectively. Then if $1 < p < \infty$ and $q=1$ or $p = \infty$ and $q > 1$, $U \subset \subset L^p( [0,T]; B )$.
\end{lemma}

We also require the extraction of weakly convergent subsequences \cite{brezis2010functional}:

\begin{lemma}[Weak Convergence of Bounded Sequences on Reflexive Banach Spaces]
	\label{thm:alao}
	Assume that $E$ is a reflexive Banach space and let $\{f_n\}$ be a bounded sequence in $E$ for $n=0,1,2,...$. Then there is a subsequence of $\{f_n\}$ converging in the weak topology of $E$.
\end{lemma}

\begin{remark}
In our application, we will be working primarily with Hilbert spaces, which are reflexive.
\end{remark}

\subsection{Basic Tools}

Let $\epsilon>0$. For $1 \leq p \leq \infty$, we define the mollification operator $\Jeps$ acting on a function $f \in L^p(\dom)$ via \begin{equation}(\Jeps [f])(x) = \sum_{k \in \mathbb{Z}^d} \e^{-\epsilon |k|^2 + 2 \pi i k \cdot x} \hat{f}(k) \label{eq:moll} \end{equation} where the discrete Fourier transform of $f$ at $k \in \mathbb{Z}^d$ is \begin{equation} \hat{f}(k) = \int_\dom \e^{-2 \pi i k \cdot x} f(x) \dd x. \label{eq:fhat} \end{equation}

The mollifier $\Jeps$ has many useful properties. We summarize important properties below \cite{greer} and transcribe proofs in section \ref{app:greer} of the Appendix.

\begin{lemma}[Mollifier Properties]
	\label{lem:moll}
	Let $d$ be the number of spatial dimensions. Given the $\Jeps$ operator defined in equation \eqref{eq:moll} we have:	
	\begin{enumerate}
		\item Let $\ell > d/2.$ Then for all $v \in \mathcal{C}^{\ell}(\dom)$, 
 $\Jeps [v] \rightarrow v$ uniformly and \\ $||\Jeps [v]||_{L^\infty(\dom)} \leq ||v||_{L^\infty(\dom)}.$ 
		\item Mollifiers commute with weak derivatives: $\forall m \in \mathbb{N} \cup \{0\}, \forall v \in H^m, \forall |\alpha| \leq m,$ $$\partial^\alpha ( \Jeps [v]) = \Jeps [\partial^\alpha v].$$
		\item For all $v, w \in \Lp{2}$, $$\int_\dom (\Jeps [v])w \dd x = \int_\dom (\Jeps [w]) v \dd x.$$ \\
		\item For all $f \in \Hm{m}$, 
		\begin{align*}
		\lim_{\epsilon \downarrow 0} ||\Jeps [f] - f||_{\Hm{m}} &= 0 \\
		||\Jeps [f] - f||_{\Hm{m-1}} &\leq \epsilon ||f||_{\Hm{m}}.
		\end{align*}
		\item Let $m > d/2$. For all $f \in \Hm{m}, \nu \in \mathbb{N} \cup \{0\}, \ell \in \{0,1,...,\nu \}, \epsilon > 0$,
		\begin{align*}
		||\Jeps [f]||_{\Hm{m+\nu}} &\lesssim_{m,\nu} \epsilon^{-\nu} ||f||_{\Hm{m}} \\
		||\Jeps [D^\nu f]||_{\Lp{\infty}} &\lesssim_{\nu} \epsilon^{m-\nu-d/2} ||f||_{\Hm{\ell}}. 
		\end{align*}
	\end{enumerate}
\end{lemma}

Often intermediate calculations make use of Sobolev spaces \cite{crime,maz2013sobolev}:

\begin{lemma}[Sobolev Inequalities and $L^p$ properties] $ $
	\label{lem:sobin}
	Let $\Omega$ be a subset of $\mathbb{R}^d$.
	\begin{enumerate}
		\item For all $m > d/2$, $H^m(\Omega)$ is a Banach algebra, i.e., for all $u, v \in H^m(\Omega)$:
		$$||u v||_{H^m(\Omega)} \lesssim ||u||_{H^m(\Omega)} ||v||_{H^m(\Omega)}.$$
		\item (Gagliardo-Nirenberg Interpolation Inequality): let $f : \Omega \to \mathbb{R}$ and let $1 \leq q, r \leq \infty$. Let $\alpha$ be a real number and $j \in \mathbb{N}$ such that
		\begin{align*}
		\frac{1}{p} &= \frac{j}{d} + (\frac{1}{r} - \frac{m}{d}) \alpha + \frac{1 - \alpha}{q} \\
		\frac{j}{m} &\leq \alpha \leq 1.
		\end{align*}
		Then there exists $C>0$ so that $$||D^j f||_{L^p(\Omega)} \leq C ||D^m f||_{L^r(\Omega)}^\alpha ||f||_{L^q(\Omega)}^{1-\alpha}.$$
		\item (Generalized H\"{o}lder's inequality): if $f_i \in L^{p_i}(\Omega)$ for $i=1,...,N$ and $\frac{1}{r} = \sum_{i=1}^N \frac{1}{p_i}$ then $||\prod_{i=1}^N f_i||_{L^r(\Omega)} \leq \prod_{i=1}^N ||f_i||_{L^{p_i}(\Omega)}.$
	\item When $m > m'$, we have that $H^m(\Omega) \subset \subset H^{m'}(\Omega)$ is a compact embedding, i.e., every bounded sequence in $H^m(\Omega)$ has an $H^{m'}(\Omega)$-convergent subsequence.
	\item (H\"{o}lder Embedding) The space $H^m(\Omega) \subset \mathcal{C}^{k,\alpha}(\Omega)$ for $m \geq \alpha + k + d/2.$
	\end{enumerate}
	
\end{lemma}

\begin{lemma}
	\label{lem:compm}
	Let $K>0$. Let $f : \mathcal{D} \subset \mathbb{R} \to \mathbb{R}$ and $g : \Omega \subset \mathbb{R}^d \to \mathcal{D}$ with $f \in \mathcal{C}^m(\mathcal{D})$, $||f||_{\mathcal{C}^m(\mathcal{D})} \leq K$, $\Omega$ bounded, and $g \in H^m(\Omega)$ for an integer $m >d/2$. Suppose that $f(0) = 0$ then 
\begin{enumerate}
\item $$||f \circ g||_{H^m(\Omega)} \lesssim_f ||g||_{H^m(\Omega)} + ||g||_{H^m(\Omega)}^{m}.$$ 
\item Furthermore, if $m>d+2$, then $$||f \circ g||_{H^m(\Omega)} \lesssim_f (1 + \ul{B} ) ||g||_{H^m(\Omega)} \lesssim (1 + \bar B) ||g||_{\Hm{m}}$$ where $\ul{B} = ||g||_{\mathcal{C}^{b^*}(\Omega)}^{m-1} \lesssim ||g||_{H^m(\Omega)}^{m-1} < \infty$, $\bar {B} = ||g||_{\mathcal{C}^{b^*}(\Omega)}^{m+1} \lesssim ||g||_{H^m(\Omega)}^{m+1} < \infty$, and $b^*$ is the largest integer less than or equal to $m/2+1.$
\end{enumerate} 
\end{lemma}

The following lemma becomes quite useful in handling the $\Hm{m}$-energy of approximating solutions.

\begin{lemma}
\label{lem:supportenergy}
Let $m>d+3$ be an integer, $\alpha$ be a multi-index with $|\alpha| \leq m$, and $Q(\cdot,\cdot)$ be bounded and continuous. Furthermore, let $\varkappa$ either be the empty set or the set containing a single order $1$ multi-index and let $\sigma$ be a nonempty set of multi-indices whereby 
\begin{itemize}
%\item $\sum_{\zeta \in \sigma} |\zeta| \leq |\alpha|+2$, 
\item $\max_{\zeta \in \sigma} |\zeta| \leq |\alpha|$, and 
\item $\sum_{\mu \in \varkappa} |\mu|+\sum_{\zeta \in \sigma} |\zeta| \leq |\alpha|+2$. 
\end{itemize} Lastly, define $b^*$ to be the largest integer not exceeding $m/2+1$ with $\bar B = ||g||_{\mathcal{C}^{b^*}(\Omega)}^{m+1}.$ Then for all $g \in \Hm{m}$,
$$|\int_{\dom} Q(g,x) \partial^{\alpha} g(x) \prod_{\mu \in \varkappa} \partial^{\mu} g(x) \prod_{\zeta \in \sigma} \partial^\zeta g(x) \dd x| \lesssim (1+ \bar B)||g||_{\Hm{m}}^2.$$
\end{lemma}

\begin{remark}
While $m>d+2$ could have worked, this assumption on $m$ is useful in Lemma \ref{cor:addedreg}.
\end{remark}

\subsection{Specialized Results for Our System}

For simplicity, we adopt the notation that $\rho u = u^+ \rho - \kappa M(\rho)$ where $$M(\rho) := \frac{u^- \rho}{1 + \rho/\tilde \rho}.$$ With this, we calculate
\begin{align}
\nabla (\rho u) &= u^+ \nabla \rho - \nabla \kappa M(\rho) - \kappa M'(\rho) \nabla \rho \label{eq:grad_rhou} \\
\lap(\rho u) &= u^+ \lap \rho - \lap \kappa M(\rho) - 2 M'(\rho) \nabla \kappa \cdot \nabla \rho - \kappa M'(\rho) \lap \rho - \kappa M''(\rho) |\nabla \rho|^2 \label{eq:lap_rhou}
\end{align}

We note other authors have studied similar equations with some monotonicity restrictions and established local existence via fixed point methods \cite{alibaud2007existence}. Owing to the unknown and possibly time-dependent sign of $\lap \kappa$, Equation \eqref{eq:lap_rhou} may fail to be monotonic in $\rho.$

\begin{lemma}[Sobolev Bounds on Model]
\label{lem:sobbounds}

	Let $\rho$ and $u$ be as in Equations \sys\ with $m>d/2$, $\rho>0$. Then:
	\begin{enumerate}
		\item $||\rho u||_{\Hm{m}} \lesssim_m 
		||\rho||_{\Hm{m}} + ||\rho||_{\Hm{m}}^{m},$
		\item $||\lap(\rho u)||_{\Hm{m}} \lesssim_m 
		||\rho||_{\Hm{m+2}}+||\rho||_{\Hm{m+2}}^{m+2},$
		\item $||\omega \rho||_{\Hm{m}} \lesssim_m 
		 ||\rho||_{\Hm{m}},$
		\item $||I[\gamma \rho u]||_{\Hm{m}} \lesssim_m  
||\rho||_{\Lp{2}} \leq ||\rho||_{\Hm{m}},$ and
		\item $||\gamma \rho u||_{\Hm{m}} \lesssim_m 
 ||\rho||_{\Hm{m}}+||\rho||_{\Hm{m}}^{m}.$
\end{enumerate}
	\end{lemma}

Part of our work in showing local existence will require Lipschitz properties of the evolution of the model. For this we use the following lemma.

\begin{lemma}[Lipschitz Bounds]
	\label{lem:lipbound}
	For $i=1,2$, let $\rho_i>0$ and $u_i$ represent a density and unattractiveness as in Equations \eqref{eq:pde}-\eqref{eq:u} with $m>d/2$. Suppose that $\rho_1$ and $\rho_2$ are also both in a bounded set $\Lambda$ in $\Hm{m}$ with $\sup_{\rho \in \Lambda} ||\rho||_{\Hm{m}} = U$. We observe $M$ and all its derivatives are bounded for positive arguments. Then:
	\begin{enumerate}
		\item $||M(\rho_1) - M(\rho_2)||_{\Hm{m}} \lesssim_{m, U} 
 ||\rho_1 - \rho_2||_{\Hm{m}} + \mathcal{O}(||\rho_1 - \rho_2||_{\Hm{m}}^2),$
		\item $||\rho_1 u_1 - \rho_2 u_2||_{\Hm{m}} \lesssim_{m, U}
||\rho_1 - \rho_2||_{\Hm{m}} + \mathcal{O}(||\rho_1 - \rho_2||_{\Hm{m}}^2),$
		\item $||\omega (\rho_1 - \rho_2)||_{\Hm{m}} \lesssim_m 
 ||\rho_1 - \rho_2||_{\Hm{m}},$
		\item $||\gamma (\rho_1 u_1 - \rho_2 u_2)||_{\Hm{m}} \lesssim_{m, U} 
||\rho_1 - \rho_2||_{\Hm{m}} + \mathcal{O}(||\rho_1 - \rho_2||_{\Hm{m}}^2),$ and
		\item $||I[\gamma (\rho_1 u_1 - \rho_2 u_2)]||_{\Hm{m}} \lesssim_{m, ||\rho_1||_{\Hm{m}}} 
||\rho_1 - \rho_2||_{\Hm{m}} + \mathcal{O}(||\rho_1 - \rho_2||_{\Hm{m}}^2).$
	\end{enumerate}
\end{lemma}

\begin{remark}
Lemmas \ref{lem:sobbounds} and \ref{lem:lipbound} apply equally well to the autonomous system.
\end{remark}

\begin{remark}
For the bounds we establish in our results throughout this paper, we will not be concerned with dependencies upon $\delta, u^-, u^+,$ or $\tilde \rho.$
\end{remark}

\section{Intermediate Results and Proof of Main Result}

\label{sec:steps}

Our proof is now done in stages.

\subsection{Step 1: Local Existence for Regularized Problem}

As a first step towards proving solution existence, we consider regularized equations. The regularized equations ensure that even with the differential operator, we map from $\Hm{m}$-functions to $\Hm{m}$-functions so that Picard can be used. We first state some results that apply equally well to the autonomous and nonautonomous systems, before restricting ourselves to the autonomous hypotheses.

\begin{proposition}
	\label{prop:bounda}
Let $\epsilon > 0$ and $m>d/2$ and define $\Aopnon : \Hm{m} \to \Hm{m}$ to be the nonlinear operator $$\Aopnon [\rho] = \delta \Jeps[ \lap ( u^+ \Jeps[\rho]  - \kappa M(\Jeps[\rho]) ) ] + \eta - \omega \rho + ( I[\gamma \rho u] - \gamma \rho u).$$ Then for all $t>0$ where $0 < \rho_1, \rho_2$ a.e., $0 < \Jeps \rho_1, \Jeps \rho_2$ with $\rho_1$ and $\rho_2$ in a bounded set $\Lambda \subset \Hm{m}$, we have  \begin{multline*}||\Aopnon[\rho_1] - \Aopnon[\rho_2]||_{H^m(\Omega)} \lesssim_{m,\epsilon,U} %\lnsim 
||\rho_1 - \rho_2||_{H^m(\Omega)} + \frac{1}{\epsilon^2} \mathcal{O}(||\rho_1 - \rho_2||_{H^m(\Omega)}^2), \quad (||\rho_1 - \rho_2||_{\Hm{m}} \downarrow 0)\end{multline*} where $\sup_{\rho \in \Lambda} ||\rho||_{\Hm{m}} = U$ %where there is a suppressed constant that depends upon $\delta, \eta, \omega, \gamma$, $\Omega$, and $||\rho_1||_{\Hm{m}}$.
\end{proposition}

\begin{proof}
	We examine the difference \begin{align*}
	\Aopnon[\rho_1] - \Aopnon[\rho_2] &= \delta \Jeps[ \lap (u^+ \Jeps[\rho_1-\rho_2] + \kappa( \Jeps[M(\Jeps[\rho_1]) - M(\Jeps[\rho_2])]) ) ] \\ &- \omega(\rho_1 - \rho_2) + I[\gamma (\rho_1 u_1 - \rho_2 u_2)] - \gamma(\rho_1 u_1 - \rho_2 u_2) 
	\end{align*}
	and consider each term. Each of the terms on the right-hand side have previously been studied in Lemma \ref{lem:lipbound} except for the Laplacian term. For that, from Lemma \ref{lem:moll}, we have
	
	\begin{align*}
	|| \Jeps [ \lap \Jeps [\rho_1-\rho_2] ]||_{\Hm{m}} &\leq || \lap \Jeps [\rho_1-\rho_2 ]||_{\Hm{m}} \\ &\leq ||\Jeps[\rho_1-\rho_2]||_{\Hm{m+2}} \\ &\lesssim_m \frac{1}{\epsilon^2} ||\rho_1 - \rho_2||_{\Hm{m}}
	\end{align*}
	
	and
	
	\begin{align*}
	& \left| \left| \Jeps \Big[ \lap \Big( \kappa (M(\Jeps[\rho_1])-M(\Jeps[\rho_2]) ) \Big) \Big] \right| \right|_{\Hm{m}} \\ 
& \lesssim ||M(\Jeps[\rho_1])-M(\Jeps[\rho_2])||_{\Hm{m+2}} \\
	& \lesssim_{m, U} ||\Jeps[\rho_1-\rho_2]||_{\Hm{m+2}} + \mathcal{O}(||\Jeps[\rho_1-\rho_2]||_{\Hm{m+2}}^2) \\ 
&\lesssim \frac{1}{\epsilon^2} ||\rho_1 - \rho_2||_{\Hm{m}} + \frac{1}{\epsilon^4} \mathcal{O}(||\Jeps[\rho_1-\rho_2]||_{\Hm{m}}^2).
	\end{align*}			
		
	In moving between the lines we decreased from $\Hm{m+2}$ to $\Hm{m}$ at the expense of acquiring an $\epsilon^{-2}$ prefactor (Lemma \ref{lem:moll}). The $U$-dependence stems from Lemma \ref{lem:lipbound}.
	 
	 As \begin{align*} ||\Aopnon[\rho_1] - \Aopnon[\rho_2]||_{\Hm{m}} &\leq ||\delta \Jeps [ \lap (u^+ \Jeps[\rho_1-\rho_2] ) ] ||_{\Hm{m}} \\ &+ ||\delta \Jeps [ \Delta ( \kappa( \Jeps[M(\Jeps[\rho_1]) - M(\Jeps[\rho_2])]) ) ] ||_{\Hm{m}} \\
&+ ||\omega(\rho_1 - \rho_2)||_{\Hm{m}} + ||I[\gamma (\rho_1 u_1 - \rho_2 u_2)]||_{\Hm{m}} \\
&+ ||\gamma(\rho_1 u_1 - \rho_2 u_2)||_{\Hm{m}}, \end{align*} we are done.
	
\end{proof}

\begin{remark}
While studying our model, it is sometimes convenient to view $\rho(t,x)$ as a function of $t$, taking values in a suitable space such as $\Hm{m}.$ We may then write $\rho(t)$ to represent $\rho(t,\cdot)$ and similarly with $\rheps$ for a regularized solution, etc.
\end{remark}

\begin{remark}
	\label{rmk:Abd}
	We also have \begin{align*} ||\mathcal{A}^\epsilon_t [\rho]||_{\Hm{m}} & \lesssim ||\eta||_{\Hm{m}} + ||\rheps||_{\Hm{m}} + ||\rheps||_{\Hm{m}}^m + ||\Jeps[\rheps]||_{\Hm{m+2}} + ||\Jeps[\rheps]||_{\Hm{m+2}}^{m+2} \\
& \lesssim_\eps ||\eta||_{\Hm{m}} + \frac{1}{\eps^2} ||\rheps||_{\Hm{m}} + \frac{1}{\eps^{2(m+2)}} ||\rheps||_{\Hm{m}}^{m+2}.
%||\Jeps \rheps||_{\Hm{m+2}} + ||\Jeps \rheps||_{\Hm{m+2}}^{m+2} \\ %+ \mathcal{O}(||\Jeps \rheps||_{\Hm{m+2}}^2) \\
%\lesssim ||\eta||_{\Hm{m}} + \frac{1}{\epsilon^2} (||\rho||_{\Hm{m}} + ||\rho||_{\Hm{m}}^{m+2}) 
\end{align*} for $\rho, \Jeps \rho > 0$ a.e. belonging to $\Hm{m}$.
\end{remark}

At this point, we focus on the autonomous system and state a theorem on local existence to regularized problems.

\begin{theorem}[Local Existence and Uniqueness for Regularized Autonomous Problem]	
	\label{thm:localreg}
	Let $\epsilon > 0$, $m>d/2$ be an integer, and $\rho_0(x) \in \Hm{m}$ be given with $\mathrm{ess} \inf_{\dom} \rho_0 = \ul{\rho}_0 > 0$. For $K>||\rho_0||_{\Hm{m}}$ arbitrary, define the open set $O_K = \{ \rho \in \Hm{m}| \quad ||\frac{1}{\rho}||_{\Lp{\infty}} < \frac{2}{\ul{\rho}_0}, ||\rho||_{\Hm{m}} < K \}.$ Then the regularized version of the autonomous system (\sysaut) given by
	\begin{equation}
	\rho^{(\epsilon)}_t = \delta \Jeps [ \lap \left(  u^+ \Jeps [\rheps] - \kappa M(\Jeps [\rheps])  \right) ] + \eta - \omega \rheps + I[\gamma \rheps \ueps] - \gamma \rheps \ueps, \label{eq:moll_pde}
	\end{equation} 	
	with $\rho(0) = \rho_0$ and where $u^{(\eps)} := u(\kappa, \rheps)$ has a unique solution $\rheps \in \mathcal{C}^1([0, T_{K,\eps,m}); O_k)$ for some $T_{K,\eps,m}>0.$
\end{theorem}

\begin{proof}
By Proposition \ref{prop:bounda}, the autonomous version of $\Aopnon, \Aopaut$, will be locally Lipschitz on $O_K$. Then by Theorem \ref{thm:picardb}, there exists $T_{K,\eps,m}>0$ so that the initial value problem ${\rheps}'(t) = \Aopaut[\rheps]$, $\rheps(0) = \rho_0$ has a solution $\rheps \in \mathcal{C}^1([0,T_{K,\eps,m}); O_K).$
\end{proof}

\begin{remark}
It is possible to establish such local existence for the nonautonomous system as well. Details appear in Appendix Section \ref{app:nom}.
\end{remark}

\subsection{Step 2: $\Hm{m}$-Energy Bounded Independent of $\epsilon$}

What we previously established is that for each $\eps>0$, regularized solutions exist and either stay in $O_K$ forever or leave the set as $t \uparrow T_{K,\eps,m}$ (Lemma \ref{thm:extb}). Unfortunately those times could depend on $\eps$ and $K$. In this step, we show it is actually possible to find an interval of existence $[0,T_m]$, independent of $K$ and $\eps,$ on which all solutions $\rheps$ exist and stay in $O_K$. As an additional bonus, we will show solutions are a.e. more regular than at $t=0$. In this section, several times are defined and Table \ref{tab:times} helps to clarify their meaning.

\begin{table}[h]
\begin{center}
\begin{tabular}{ l | p{5.5in} }
\bf Time & \bf Meaning \\
\hline
& \\
$T_{K,\eps,m}$ & interval of local solutions: depends on $\eps$, the set $O_K$, and $m$ \\
\hline
& \\
$\varpi_m$ & guaranteed interval of existence for $E_m(t)$ satisfying a differential inequality for the $\Hm{m}$-energy, where $E_m(t)$ treated as independent of $\rheps$ \\
\hline
& \\
$\tilde T_{K, \eps, m}$ & minimum of $T_{k,\eps,m}$ and $\varpi_m$ \\
\hline
& \\
$\tilde T_{0,m}$ & time where solution is guaranteed to be bounded below suitably \\
\hline
& \\
$T_m$ & the minimum of $\varpi_m$ and $\tilde T_{0,m}$ after suitable $K$ chosen
\end{tabular}
\caption{List of times in analysis.} \label{tab:times}
\end{center}
\end{table}

Our first main result is Proposition \ref{prop:Ebound}. 

\begin{proposition}[Control of Energy]
	\label{prop:Ebound}
Let $\rheps$ be defined as in Theorem \ref{thm:localreg} and let $m>d^* := 3+d$ be an integer. Then there exists a time $\varpi_m > 0$ and a continuous function $E_m(t)$ defined on $[0, \varpi_m]$ with $\xi_m := ||E_m||_{L^\infty([0,\varpi_m])} < \infty$ whereby
\begin{enumerate}
\item for all $\eps>0$, $\frac{1}{2} ||\rheps||_{\Hm{m}}^2 \leq E_m(t)$ on $[0,\min\{\varpi_m,T_{K,\eps,m} \})$ and
\item $\frac{1}{2} ||\rheps||_{\Hm{m}}^2 + \int_0^t ||D^{m+1} \Jeps[\rheps]||_{\Lp{2}}^2 \dd s \lesssim \xi_m + (\xi_m^{1/2} + \xi_m^{(m+3)/2}) t < \infty$ for $t \in [0,\min\{ \varpi_m, T_{K,\eps,m} \}).$ Moreover,
\item on $[0,\min\{ \varpi_m, T_{K,\eps,m} \})$, $E_m^{(\epsilon)}(t) := \frac{1}{2} ||\rheps||_{\Hm{m}}^2$ satisfies \begin{align*}
{E_m^{(\epsilon)}}'(t) &\lesssim_{m} {E_m^{(\eps)}}^{1/2} + (1+\bar B) E_m^{(\eps)} \\
&\lesssim ( {E_m^{(\epsilon)}}^{1/2} + {E_m^{(\epsilon)}}^{(m+3)/2} ),\end{align*} where $\bar B \sim ||\rheps||_{\mathcal{C}^{b^*}(\dom)}^{m+1}$ and $b^*$ is the largest integer not exceeding $m/2+1$. 
\item We also have that for some $C>0$, for all $\eps>0$, 
\begin{equation}
E_m^{(\eps)}(t) \leq E_m(t) = \begin{cases} (E_m^{(\epsilon)}(0)^{1/2}+ 2 C t)^2, \quad E_m^{(\eps)}(0) < 1 \\
	({E_m^{(\epsilon)}(0)}^{-(m+1)/2} - (m+1) C t)^{-2/(m+1)}, \quad E_m^{(\eps)}(0) \geq 1 \end{cases} \leq \xi_m 
	< \infty \label{eq:Eintbd}
	\end{equation} on $[0,\min\{\varpi_m,T_{K,\eps,m} \})$. Also, $E_m^{(\eps)}(0) = E_m(0) = \frac{1}{2} ||\rho_0||_{\Hm{m}}^2.$
\end{enumerate}
\end{proposition}

\begin{proof}[Proof of Proposition \ref{prop:Ebound}]
	We first note that if $m > d+3$ 
and $\rheps \in \Hm{m}$. % then %\begin{align*}||D^2 \rheps||_{\Lp{\infty}} \overbrace{\lesssim}^{m-2>d/2} ||\rheps||_{\Hm{m}}. \end{align*}  
We define $\sigma^*(m) = m/2+1$ and note that all derivative orders $j$ with $j \leq \sigma^*(m)$ for $\Jeps[\rheps]$ and $\rheps$ are in $\Lp{\infty}$ and bounded by $||\rheps||_{\Hm{m}}$ --- from Lemma \ref{lem:moll}, $||D^j \Jeps[\rheps]||_{\Lp{\infty}} \lesssim  ||\Jeps[\rheps]||_{\Hm{m}} \lesssim ||\rheps||_{\Hm{m}}$. We shall denote $b^*$ to be the largest integer less than or equal to $\sigma^*$ and $\ul{B} =  ||\rheps||_{\mathcal{C}^{b^*}(\dom)}^{m-1}.$

We define $E_m^{(\epsilon)}(t) = \frac{1}{2} ||\rheps||_{\Hm{m}}^2$ and consider ${E_m^{(\epsilon)}}'(t) = \sum_{|\alpha|\leq m} \int_{\dom} \partial^\alpha \rheps \partial^\alpha \rheps_t \dd x.$ Let $\alpha$ be a multi-index. We use Lemmas \ref{lem:compm} and \ref{lem:sobbounds} and to find	
	\begin{align}
	{E_m^{(\epsilon)}}'(t) &= \sum_{|\alpha|\leq m} \int_{\dom} \partial^\alpha \rheps \partial^\alpha \rheps_t \dd x \nonumber \\
	&= \sum_{|\alpha|\leq m} \int_{\dom} \partial^\alpha \rheps \partial^\alpha \Big( \Jeps \Big[ \delta \lap ( u^+ \Jeps[\rheps] \nonumber \\ &- \kappa M(\Jeps[\rheps]) )  \Big] + \eta - \omega \rheps + I[\gamma \rheps \ueps] - \gamma \rho \ueps \Big) \dd x \nonumber \\
	&\leq \sum_{|\alpha|\leq m} \delta \int_{\dom} \partial^\alpha \rheps \partial^\alpha \Jeps[ \lap ( u^+ \Jeps [\rheps] - \kappa M(\Jeps [\rheps]) ) ] \dd x \nonumber \\&+ ||\rheps||_{\Hm{m}} \Big( ||\eta||_{\Hm{m}} + ||\omega \rheps||_{\Hm{m}} \nonumber \\ &+ ||I[\gamma \rheps \ueps]||_{\Hm{m}} + ||\gamma \rheps \ueps||_{\Hm{m}} \Big) \nonumber \\
&\lesssim ||\rheps||_{\Hm{m}} + (1+\ul{B})||\rheps||_{\Hm{m}}^2 \nonumber \\ &+  \sum_{|\alpha|\leq m} \delta \int_{\dom} \partial^\alpha \rheps \partial^\alpha \Jeps[\lap (u^+ \Jeps [\rheps] - \kappa M(\Jeps [\rheps]) )] \dd x \label{eq:EB1} \\
	&\lesssim_m \sum_{|\alpha|\leq m} \int_{\dom} \partial^\alpha \rheps \partial^\alpha \Jeps[\lap (u^+ \Jeps [\rheps] - \kappa M(\Jeps [\rheps]) )] \dd x \nonumber \\ &+ ||\rheps||_{\Hm{m}} + ||\rheps||_{\Hm{m}}^{m+1}. \label{eq:Em1}
	\end{align}
	We used that $||\gamma \rheps \ueps||_{\Hm{m}} \leq ||\gamma||_{\Hm{m}} ||\rheps u^+ - \kappa M(\rheps)||_{\Hm{m}} \lesssim ||\rheps||_{\Hm{m}} + (1+ \ul{B} ) ||\rheps||_{\Hm{m}}.$

	We need to focus our attention on the term
	\begin{align}
X &= \sum_{|\alpha|\leq m} \int_{\dom} \partial^\alpha \rheps \partial^\alpha \Jeps[ \lap ( u^+ \Jeps [\rheps] - \kappa M(\Jeps [\rheps]) ) ] \dd x \nonumber \\
&= \sum_{|\alpha|\leq m} \int_{\dom} \partial^\alpha \rheps \Jeps[ \partial^\alpha \big( \lap ( u^+ \Jeps [\rheps] - \kappa M(\Jeps [\rheps]) ) \big) ] \dd x \nonumber \\
&= \sum_{|\alpha|\leq m} \int_{\dom} \Jeps[ \partial^\alpha \rheps]  \partial^\alpha \big( \lap ( u^+ \Jeps [\rheps] - \kappa M(\Jeps [\rheps]) ) \big)  \dd x \nonumber \\
&= \sum_{|\alpha|\leq m} \int_{\dom} \partial^\alpha \Jeps [\rheps]  \nabla \cdot \nabla \partial^\alpha \big( u^+ \Jeps [\rheps] - \kappa M(\Jeps [\rheps]) \big)  \dd x \nonumber \\
&= -\sum_{|\alpha|\leq m} \int_{\dom} \nabla \partial^\alpha \Jeps [\rheps]  \cdot \nabla \partial^\alpha \big( u^+ \Jeps [\rheps] - \kappa M(\Jeps [\rheps]) \big)  \dd x
	\end{align}
	We made use of the argument of the outer $\Jeps$ being in $\Hm{\infty}$ (i.e. in $\Hm{m}$ for all $m$), the $\Lp{2}$-properties of the factors, and $\rheps \in \Hm{m}$.

 For any $|\alpha|<m$ we have the terms of $X$ are	
	\begin{align*}
	& \leq ||\Jeps [\rheps]||_{\Hm{m}} 
	\times (u^+ ||\Jeps [\rheps]||_{\Hm{m}} + ||\kappa M(\Jeps [\rheps])||_{\Hm{m}}) \\
	&\leq ||\Jeps [\rheps]||_{\Hm{m}} \times (u^+ ||\Jeps [\rheps]||_{\Hm{m}} + ||\kappa||_{\Hm{m}} ||M(\Jeps [\rheps])||_{\Hm{m}}) \\
	&\lesssim ||\Jeps [\rheps]||_{\Hm{m}}^2 + (1+\ul{B} ) ||\Jeps [\rheps]||_{\Hm{m}}^2 \\
	&\lesssim ||\rheps||_{\Hm{m}}^2 + (1+ \ul{B} ) ||\rheps||_{\Hm{m}}^2 \\
	&\lesssim ||\rho||_{\Hm{m}}^2 + ||\rho||_{\Hm{m}}^{m+1} .
	\end{align*}
	
	And so we only need to concern ourselves with the case $|\alpha| = m.$ Then 	
	\begin{align*}
& -\int_{\Td} \nabla  \partial^\alpha \Jeps [\rheps] \cdot \nabla \partial^\alpha \big(u^+ \Jeps[\rheps] - \kappa M(\Jeps[\rheps])\big) \dd x \\
	&= -\sum_i \int_{\dom} \partial^{\alpha+e_i} \Jeps [\rheps] \partial^{\alpha+e_i} \big(u^+ \Jeps [\rheps] - \kappa M(\Jeps [\rheps])\big) \dd x \\
	&= -\sum_i \int_{\dom} \Bigg[ \partial^{\alpha+e_i} \Jeps [\rheps] \left( u^+ \partial^{\alpha+e_i} \Jeps [\rheps] \vphantom{\sum_{\sigma \in \mathcal{P}(\alpha+e_i-\beta)} M^{(|\sigma|)} (\Jeps[\rho]) \prod_{\zeta \in \sigma} \partial^\zeta \Jeps \rheps} \right. \\
	&\left. - \sum_{\beta \leq \alpha+e_i} { \alpha+e_i \choose \beta } \partial^\beta \kappa \left( \sum_{\sigma \in \mathcal{P}(\alpha+e_i-\beta)} M^{(|\sigma|)} (\Jeps[\rheps]) \prod_{\zeta \in \sigma} \partial^\zeta \Jeps [\rheps] \right) \right) \Bigg] \dd x. 
	\end{align*}
At this point, we consider multiple cases upon the terms of the integrand being summed over $\beta$ and $\sigma$. The form, $F$, of many of the terms being summed can be represented as
$$F = \int_{\dom} Q(\Jeps[\rheps],x) \partial^{\alpha+e_i} \Jeps[\rheps] \prod_{\zeta \in \sigma} \partial^\zeta \Jeps[\rheps],$$ where $\sum_{\zeta \in \sigma} |\zeta| = m+1-|\beta|$.

	\begin{itemize}
	\item {\bf case 1:} {\it all terms being summed where $|\beta| \geq 2$}, {\it after a single integration by parts}, take the form 
\begin{align*} C_1 &:= \int_\dom Q_1(\Jeps \rheps, x) \partial^\alpha \Jeps [\rheps] \partial^{e_i} \Jeps[\rheps] \prod_{\zeta \in \sigma} \partial^\zeta \Jeps [\rheps] \dd x \\ &+ \int_\dom Q_2(\Jeps [\rheps], x) \partial^\alpha \Jeps [\rheps]  \prod_{\zeta \in \sigma} \partial^\zeta \Jeps [\rheps] \dd x \\
&+ \int_\dom Q_3(\Jeps [\rheps], x) \partial^\alpha \Jeps [\rheps] \prod_{\zeta \in \tilde \sigma} \partial^\zeta \Jeps [\rheps] \dd x,	\end{align*} 
where $|\alpha| = m$, $\sum_{\zeta \in \tilde \sigma} |\zeta| \leq m$, $|\tilde \sigma| \leq m-1$, and $Q_1$, $Q_2$, and $Q_3$ denote smooth functions in their arguments. Then, by Lemma \ref{lem:supportenergy}, we conclude that $$C_1 \lesssim (1+ \bar B ) ||\rheps||_{\Hm{m}}^2 \lesssim ||\rheps||_{\Hm{m}}^2 + ||\rheps||_{\Hm{m}}^{m+3}.$$

	\item {\bf case 2:} {\it all terms being summed where $|\beta|=1$ and $|\sigma| = 1$} take the form $$C_2:= \int_\dom \partial^{\alpha+e_i} \Jeps [\rheps] \partial^{e_j} \kappa M'(\Jeps [\rheps]) \partial^{\alpha+e_i-e_j} \Jeps [\rheps] \dd x,$$ where $j$ is some value in $\{1, 2, ..., d \}.$

The terms can be integrated by parts once to furnish
	\begin{align*}
	C_2 &= \int_\dom \frac{1}{2} (\partial^{\alpha+e_i-e_j} \Jeps \rheps)^2 \left( \partial^{2e_j} \kappa M'(\Jeps \rheps) + \partial^{e_j} \kappa M''(\Jeps \rheps)  \partial^{e_j} \rheps \right) \dd x \\
	&\leq \frac{1}{2} ||\partial^{2e_j} \kappa M'(\Jeps \rheps) + \partial^{e_j} \kappa M''(\Jeps \rheps)  \partial^{e_j} \rheps||_{\Lp{\infty}} \\ &\times \int_\dom (\partial^{\alpha+e_i-e_j} \Jeps \rheps)^2 \dd x \\
	& \lesssim (1+ \bar B^{1/(m+1)}) ||\rheps||_{\Hm{m}}^2 \\
	& \lesssim (1 + ||\rheps||_{\Hm{m}}) ||\rheps||_{\Hm{m}}^2 \\
	&= ||\rheps||_{\Hm{m}}^2 + ||\rheps||_{\Hm{m}}^3,
	\end{align*}
	where the $\bar B$-term comes from the boundedness of the terms in the norm, with $D^1 \rheps$ bounded by $\bar B^{1/(m+1)}.$

	\item {\bf case 3:} {\it all terms being summed where $|\beta|=1$ and $m \geq |\sigma| \geq 2$ after a single integration by parts} take the form 
	\begin{align*} C_3 &:= \int_\dom Q_1(\Jeps \rheps, x) \partial^\alpha \Jeps [\rheps] \prod_{\zeta \in \sigma} \partial^\zeta \Jeps [\rheps] \dd x \\ &+ \int_\dom Q_2(\Jeps [\rheps], x) \partial^\alpha \Jeps [\rheps] \partial^{e_i} \Jeps [\rheps] \prod_{\zeta \in \sigma} \partial^\zeta \Jeps [\rheps] \dd x \\
&+ \int_\dom Q_3(\Jeps [\rheps], x) \partial^\alpha \Jeps [\rheps] \prod_{\zeta \in \tilde \sigma} \partial^\zeta \Jeps [\rheps] \dd x,
	\end{align*} where $|\alpha|=m$, $\sum_{\zeta \in \tilde \sigma} |\zeta| = m+1$, $|\tilde \sigma| \leq m$, and $Q_1$, $Q_2$, and $Q_3$ denote smooth functions in their arguments. Then by Lemma \ref{lem:supportenergy}, we conclude that $$C_3 \lesssim (1+ \bar B) ||\rheps||_{\Hm{m}}^2 \lesssim ||\rheps||_{\Hm{m}}^2 + ||\rheps||_{\Hm{m}}^{m+3}.$$

	\item {\bf case 4:} {\it the terms where $|\beta|=0$ and $|\sigma|=1$ can be grouped with the term involving no summation over $\beta$ or $\sigma$} to obtain
	\begin{align*} C_4 &:= - \int_\dom \partial^{\alpha+e_i} \Jeps [\rheps] ( u^+ \partial^{\alpha+e_i} \Jeps [\rheps] - \kappa M'(\Jeps [\rheps]) \\
&\times \partial^{\alpha+e_i} \Jeps[\rheps]) \dd x.\end{align*}

The terms can be rearranged in a more transparent fashion:
	\begin{align*}
	C_4 = - \int_\dom (u^+ - \kappa M'(\Jeps [\rheps])) (\partial^{\alpha+e_i} \Jeps [\rheps])^2 \dd x \leq 0
	\end{align*}
	since $u^+ - \kappa M' > 0$.

	\item {\bf case 5:} {\it the terms where $|\beta|=0$, $\max_{\zeta \in \sigma} |\zeta| = m$ (so that $|\sigma|=2$)} take the form $$C_5 := \int_\dom \partial^{\alpha+e_i} \Jeps [\rheps] \kappa M''(\Jeps [\rheps]) \partial^{\alpha+e_i-e_j} \Jeps [\rheps] \partial^{e_j} \Jeps [\rheps] \dd x,$$ for some $j \in \{1, 2, ..., d \}.$ 

We integrate by parts to obtain terms of form
	\begin{align*}
	C_5 &= \int_\dom \frac{1}{2} (\partial^{\alpha+e_i-e_j} \Jeps [\rheps])^2 \\
	&\times \left( \partial^{e_j} \kappa M''(\Jeps [\rheps]) \partial^{e_j} \Jeps [\rheps] + \kappa M'''(\Jeps [\rheps]) (\partial^{e_j} \Jeps [\rheps])^2 \right. \\
	&+ \left. \kappa M''(\Jeps [\rheps]) \partial^{2e_j} \Jeps [\rheps] \right) \dd x \\
	&\leq \frac{1}{2} ||\partial^{e_j} \kappa M''(\Jeps [\rheps]) \partial^{e_j} \Jeps [\rheps] + \kappa M'''(\Jeps [\rheps]) (\partial^{e_j} \Jeps [\rheps])^2 \\ &+ \kappa M''(\Jeps [\rheps]) \partial^{2e_j} \Jeps [\rheps]||_{\Lp{\infty}} \times \int_\dom (\partial^{\alpha+e_i-e_j} \Jeps [\rheps])^2 \dd x \\
	&\lesssim (\bar B^{1/(m+1)} + \bar B^{2/(m+1)}) ||\rheps||_{\Hm{m}}^2 \\
	&\lesssim (||\rheps||_{\Hm{m}} + ||\rheps||_{\Hm{m}}^2) ||\rheps||_{\Hm{m}}^2 \\
	&= ||\rheps||_{\Hm{m}}^3 + ||\rheps||_{\Hm{m}}^4.
	\end{align*}

	\item {\bf case 6:} {\it the terms where $|\beta|=0$, $\max_{\zeta \in \sigma} |\zeta| \leq m-1$ (so that $|\sigma| \geq 2$ and $\sum_{\zeta \in \sigma} |\zeta| = m+1$) after a single integration by parts} take the form
	\begin{align*} C_6 &:=
	\int_\dom Q_1(\Jeps [\rheps], x) \partial^\alpha \Jeps [\rheps] \prod_{\zeta \in \sigma} \partial^\zeta \rho \dd x \\ &+ \int_\dom Q_2(\Jeps [\rheps], x) \partial^\alpha \Jeps [\rheps] \partial^{e_i} \Jeps [\rheps] \prod_{\zeta \in \sigma} \partial^\zeta \Jeps[\rho] \dd x \\ &+ \int_\dom Q_3(\Jeps [\rheps], x) \partial^\alpha \Jeps \rheps \prod_{\zeta \in \tilde \sigma} \partial^\zeta \Jeps [\rheps] \dd x,
	\end{align*}
	where $\sum_{\zeta \in \tilde \sigma} |\zeta| = m+2$, $|\tilde \sigma| = |\sigma|$, and $Q_1$, $Q_2$, and $Q_3$ are smooth functions in their arguments. In this case, through Lemma \ref{lem:supportenergy}, we find $$C_6 \lesssim (1+\bar B)||\rheps||_{\Hm{m}}^2 \lesssim ||\rheps||_{\Hm{m}}^2 + ||\rheps||_{\Hm{m}}^{m+3}.$$

	\end{itemize}

	From all preceding cases, we conclude 
\begin{align}{E_m^{(\eps)}}'(t) &\lesssim||\rheps||_{\Hm{m}} + (1+ \bar B) ||\rheps||_{\Hm{m}}^2 \\
&\lesssim ||\rheps||_{\Hm{m}} + ||\rheps||_{\Hm{m}}^{m+3} 
\end{align}
 so that there is a $C>0$, independent of $\epsilon$, where	
	\begin{equation}
	{E_m^{(\epsilon)}}'(t) \leq C \left( {E_m^{(\epsilon)}}^{1/2} + {E_m^{(\epsilon)}}^{(m+3)/2} \right). \label{eq:Ebound}
	\end{equation}
We used $1 + \ul{B} \lesssim 1 + \bar B.$

	While $E_m^{(\eps)}$ was originally defined in terms of the $\Hm{m}$-norm of $\rheps$, we can also consider it a function in isolation evolving according to the differential inequality \eqref{eq:Ebound}. For the rest of the analysis, we treat $E_m^{(\eps)}$ as such.
	
	We note that if $E_m^{(\epsilon)}(0) < 1$ then {\it until} $E_m^{(\epsilon)}$ reaches $1$ (if it does), $${E_m^{(\epsilon)}}'(t) \leq 2 C {E_m^{(\epsilon)}}^{1/2}.$$ Also, if $E_m^{(\epsilon)}(0) \geq 1$ then (in the worst case where it stays at or above 1) $${E_m^{(\epsilon)}}'(t) \leq 2 C {E_m^{(\epsilon)}}^{(m+3)/2}.$$ Thus,
	\begin{equation}
	E_m^{(\eps)}(t) \leq E_m(t) := \begin{cases} ( E_m^{(\epsilon)}(0)^{1/2} + 2 C t)^2, \quad E_m^{(\eps)}(0) < 1 \\
	({E_m^{(\epsilon)}(0)}^{-(m+1)/2} - (m+1) C t)^{-2/(m+1)}, \quad E_m^{(\eps)}(0) \geq 1 \end{cases} < \infty
	\end{equation} over $[0, T_E)$ where
	\begin{equation}
	T_E = \begin{cases} \frac{1 - E_m^\epsilon(0)^{1/2}}{2C}, \quad E_m^{(\epsilon)}(0) < 1 \\
	\frac{ {E_m^{(\epsilon)}(0)}^{-(m+1)/2} }{(m+1)C}, \quad E_m^{(\epsilon)}(0) \geq 1.
	\end{cases}
	\end{equation}

	Choosing $0 < \varpi_m < T_E$ completes the proof of (1), (3), and (4), and while $E_m^{(\eps)}$ could blow up in finite time, it is finite over $[0,\varpi_m].$ The time $\varpi_m$ may depend upon $m$ but not $\eps.$ 

	By moving the terms of case 4 to the left side of the overall inequality we built for ${E_m^{(\eps)}}'(t)$, we can actually show that ${E_m^{(\eps)}}'(t) + ||D^{m+1} \Jeps[\rheps] ||_{\Lp{2}}^2 \lesssim {E_m^{(\eps)}}^{1/2} + {E_m^{(\eps)}}^{(m+3)/2}$. Then integrating, we have ${E_m^{(\eps)}}(t) + \int_0^{t} ||D^{m+1} \Jeps [\rheps]||_{\Lp{2}}^2 \dd s \lesssim \xi_m + (\xi_m^{1/2} + \xi_m^{(m+3)/2}) t < \infty$ on $[0,\min \{ \varpi_m, T_{K,\eps,m} \})$. This comes from the fact that $u^+ - \kappa M'$ is strictly bounded below by a positive constant. Thus item (2) is established.

\end{proof}

\begin{remark}
	The function $E_m(t)$ obtained may depend upon $m$. As the regularity we enforce upon the solutions increases, the shorter the interval over which a bound exists may become. The initial condition may in fact be more regular than $\Hm{m}.$
\end{remark}

\begin{remark}
	The $\min\{ \varpi_m, T_{K,\eps,m} \}$ is necessary here because at present we do now know whether the solution $\rheps$ leaves the set $O_K$ before $\varpi_m.$ It's also possible the function $E_m(t)$ ceases to exist (it is a worst case upper bound) but solutions themselves still exist up to $t=T_{K,\eps,m}.$
\end{remark}

\begin{corollary}
\label{cor:Hmp1}
Let $\rheps$, $\xi_m$, and $E_m$ be defined as in Proposition \ref{prop:Ebound} with $m>d^* := 3+d$ an integer. Then $\rheps \in L^2([0,\tilde T_{K,\eps,m}]; \Hm{m+1})$ where $\tilde T_{K,\eps,m} := \min\{\varpi_m,T_{K,\eps,m}\}.$ Also, $||\rheps||_{\Hm{m+1}}(t) < \infty$ for a.e. $t \in [0, \tilde T_{K,\eps,m}].$
\end{corollary}

\begin{proof}
The $L^2([0,\tilde T_{K,\eps,m}]; \Hm{m+1})$-part part can be gleaned from Proposition \ref{prop:Ebound}. % in section \ref{app:exist} of the Appendix.
 In particular it follows that $\int_0^{\tilde T_{K,\eps,m}} ||D^{m+1} \Jeps[\rheps]||_{L^2(\dom)}^2 \dd s \leq \xi_m + ( \xi_m^{1/2} + \xi_m^{(m+3)/2}) \tilde T_{K,\eps,m}  < \infty$. Taking $\eps \downarrow 0$, we see that indeed, $D^{m+1} \rheps \in L^2([0,\tilde T_{K,\eps,m}]; L^2(\dom))$, which means $\rheps \in L^2([0,\tilde T_{K,\eps,m}]; \Hm{m+1})$ or else the bound could not be independent of $\eps.$ Then, we must have that $\rheps \in \Hm{m+1}$ for a.e. $0 \leq t \leq \tilde T_{K,\eps,m}.$
\end{proof}

With parameters like $\tilde T_{K,\eps,m}$, there are a lot of dependencies. The next proposition allows us to break those dependencies.

\begin{proposition}[Reducing Dependencies]
\label{prop:red}
Consider the setup of Theorem \ref{thm:localreg} where $m>d^* := 3+d$ is an integer, $E_m$ is the function found in Proposition \ref{prop:Ebound} with $\xi_{m} = ||E_m||_{L^\infty([0,\varpi_m])}$, and choose $K \sim \sqrt{\xi_m}$ so that even if a solution lasted in $O_K$ up to time $\varpi_m$, $||\rheps||_{\Hm{m}} < K/2.$ Let the solutions $\rheps \in \mathcal{C}^1([0,T_{K,\eps,m}); O_K)$ have a maximal interval of existence, $[0, T_{K,\eps,m} )$ with $\eps>0$ arbitrary. Then
\begin{itemize}
\item There is a $K^*$, depending on $K$, so that on $[0, T_{K,\eps,m})$, $||\rheps_t||_{\Lp{\infty}} < K^*$ and the minimum value of $\rheps$ is bigger than or equal to $\ul{\rho}_0 - K^* t$.
\item For the maximal interval of existence of solutions $\rheps \in \mathcal{C}^1([0,T_{K,\eps,m}); O_K)$ we have $T_{K,\eps,m} > T_m := \min \{ \varpi_m, \tilde T_{0,m} \}$ where $\tilde T_{0,m} := \frac{\ul{\rho}_0}{4 K^*}.$
\end{itemize}
\end{proposition}

\begin{proof}
By Theorem \ref{thm:localreg}, we have a solution $\rheps \in \mathcal{C}^1 ( [0, T_{K,\epsilon,m}); \Hm{m} )$ with $\rheps(t) \in \Hm{m} \subset \mathcal{C}^{2}(\dom).$ We study the operator $\Aopaut$ and note that 
\begin{align*}
||\Aopaut[\rheps]||_{\Lp{\infty}} &= ||\delta \Jeps \lap ({\Jeps} [ u^+ \rheps - \kappa M( \Jeps[\rheps] ) ] ) + \eta - \omega \rheps \\
&+ ( I[\gamma \rheps \ueps] - \gamma \rheps \ueps)||_{\Lp{\infty}} \\
&\lesssim ||\eta||_{\Lp{\infty}} + ||\omega \rheps||_{\Hm{m}} + ||\gamma \rheps \ueps||_{\Hm{m}} + ||I[\gamma \rheps \ueps]||_{\Hm{m}} \\
&+ || \lap ( {\Jeps} [\rheps] - \kappa M(\Jeps[\rheps]) ) ||_{\Lp{\infty}} \\
&\lesssim ||\eta||_{\Lp{\infty}} + ||\omega \rheps||_{\Hm{m}} + ||\gamma \rheps \ueps||_{\Hm{m}} + ||I[\gamma \rheps \ueps]||_{\Hm{m}} \\
&+ || {\Jeps} [\rheps] - \kappa M(\Jeps[\rheps]) ||_{\Hm{m}} \\
&\lesssim 1 + ||\rheps||_{\Hm{m}} + ||\rheps||_{\Hm{m}}^m. 
\end{align*} 
The first inequality comes from the Triangle Inequality and bounding $||\cdot||_{\Lp{\infty}}$ of a smooth argument by the $||\cdot||_{\Lp{\infty}}$ of that argument; the next inequality follows from Sobolev embedding with $m>d/2+2$. Then with $\rheps \in \Hm{m}$, along with Lemmas \ref{lem:moll}, \ref{lem:compm} and \ref{lem:sobbounds}, the final inequality follows.

On $[0, T_{K,\eps,m})$, $||\rheps||_{\Hm{m}} < K$ and so we have some $K^*>0$ so that $||\Aopaut||_{\Lp{\infty}} \leq K^*$. Then for $x \in \Td$, $| \rheps_t(t,x) | \leq K^*$ and thus $\rheps(t,x) \geq \ul{\rho}_0 - K^* t$, which yields item 1.

To handle item 2, suppose, by way of contradiction, that $T_{K,\eps,m} \leq T_m.$ We first note that on $[0, T_{K,\eps,m})$ the solutions are in $O_K$ and thus $\frac{1}{\rheps} \leq \frac{1}{\ul{\rho}_0 - K^* t} \leq \frac{1}{\ul{\rho}_0 - K^* T_{K,\eps,m}} \leq \frac{1}{\ul{\rho}_0 - K^* \tilde T_{0,m}} \leq \frac{4}{3 \ul{\rho}_0}.$ Also by choice of $K$, since $T_{K,\eps,m} \leq \varpi_m$, we have $||\rheps||_{\Hm{m}} < K/2.$ 

By the Extension Lemma (Lemma \ref{thm:extb}), with finite $T_{K,\eps,m}$, either $\lim_{t \uparrow T_{K,\eps,m}} ||\rheps||_{\Hm{m}} = K$ or $\lim_{t \uparrow T_{K,\eps,m}} ||\frac{1}{\rheps}||_{\Lp{\infty}} = \frac{2}{\ul{\rho}_0}.$ Neither is possible given the bounds above, whence $T_{K,\eps,m} > T_m.$
\end{proof}

The remaining two lemmas in this step can be used to gain one more order of regularity in solutions after $t=0.$

\begin{lemma}
\label{lem:Emexist}
Let $\rheps$, $\bar B$, $E_m^{(\eps)}$, and $b^*$ be defined as in Proposition \ref{prop:Ebound} with $m>d^* := 3+d$ an integer. Then as long as $||\rheps||_{\mathcal{C}^{b^*}(\dom)} < \infty$, ${E_m^{(\epsilon)}}$ remains finite.
\end{lemma}

\begin{proof}
If $E_m^{(\eps)} \leq 1$ then there is nothing to prove. Otherwise, if $E_m^{(\eps)} > 1$, we have that ${E_m^{(\eps)}}'(t) \leq 2 C (1+ \bar B) {E_m^{(\eps)}}$ for some $C>0.$ While $E_m^{(\eps)}>1$, we have from Gr\"onwall that $$E_m^{(\eps)}(t) \leq E_m^{(\eps)}(0) \exp(\int_0^t 2 C (1+\bar B(s)) \dd s).$$ Thus, as long as $\bar B$ remains finite, we are guaranteed a finite energy. 
\end{proof}

\begin{lemma}
\label{cor:addedreg}
Let $\rheps$ and $T_m$ defined as in Proposition \ref{prop:red} (note that $m>d^* := 3+d$). Then $\rheps \in \mathcal{C}((0,T_m]; \Hm{m+1}).$
\end{lemma}

\begin{proof}
From $\rheps \in \Hm{m+1}$ a.e. on $[0,T_m]$ (Corollary \ref{cor:Hmp1}) we can see that for each time point $s \in (0, T_m]$ (where $\rheps \in \Hm{m}$), there is a time $s'<s$ where $\rheps(s') \in \Hm{m+1}.$ Consider the ODE system for a function $\varrho$ with $\varrho (t=s') = \rheps(s') \in \Hm{m+1}$ as an initial condition. Then the same analysis can be done to find a solution $\varrho \in \mathcal{C}^1([s',s'']; \Hm{m+1})$ for some $s'' > s'.$ Furthermore, the energy $E_{m+1}^{(\eps)}$ is bounded provided $||\rheps||_{\mathcal{C}^{b^{**}}}$ is finite, where $b^{**}$ is the largest integer not exceeding $(m+1)/2+1.$ Since $m>3+d$, $||\rheps||_{\mathcal{C}^{b^{**}}}$ is finite on $[0,T_m].$ This stems from the fact that $||\rheps||_{\Hm{m}}$ is finite on $[0, T_m].$ Also, the solution stays bounded below by $\ul{\rho}_0/2.$ By uniqueness of solutions, $\varrho = \rheps$ on $[s',s'']$ with $s'' \geq T_m$ and we have $\rheps \in \mathcal{C}^1((0,T_m]; \Hm{m+1}).$
\end{proof}

\subsection{Step 3: Extracting a Convergent Subsequence in Weaker Space}

We ultimately wish to take a subsequence of $\rheps$ converging to a solution with our desired properties. Through Lemma \ref{lem:weakConverge}, we firstly show that a subsequence can be extracted that converges weakly to an element of $L^\infty([0,T_m]; \Hm{m}).$ The weak convegence on its own is insufficient and we thus also find a subsequence that converges strongly, but in a weaker space. 
Owing to the nonlinearities in the derivative terms, maneuvering the equations in such a way as to prove $\rheps$ is Cauchy in some weaker space as $\epsilon \downarrow 0$ is difficult. However, we really only require that a subsequence converges and for that we can use Aubin-Lions-Dubinski{\u\i} (Lemma \ref{thm:aubin}).

%{\color{red}
\begin{lemma}[Weakly Convergent Subsequence]
\label{lem:weakConverge}
Let $\rheps$ be defined as in Theorem \ref{thm:localreg} and let $m > 3+d$ where $m$ is an integer. Then there is a subsequence of $\rheps$ converging weakly in $L^2([0,T_m]; \Hm{m})$ to $\rho^* \in L^\infty([0,T_m]; \Hm{m}).$
\end{lemma}

\begin{proof}
From Proposition \ref{prop:Ebound}, $\rheps$ is bounded in $L^\infty([0,T_m]; \Hm{m})$ by some $K>0$ for all $\eps>0$ and thus the sequence is bounded in $L^2([0,T_m]; \Hm{m}).$ By Lemma \ref{thm:alao}, there is a weakly convergent subsequence $\rho^{(\eps')}$ converging to some $\rho^* \in L^2([0,T_m]; \Hm{m}).$ 

This forces $\rho^* \in L^\infty([0,T_m]; \Hm{m}).$ Were this not the case then let $Z>2 K$ and let the set $T_Z = \{ t | \quad ||\rho^*||_{\Hm{m}} \geq Z \}.$ We have that $|T_Z|>0$. Let us denote $f(t,x) = \chi_{T_Z}(t) \rho^*(t,x) \in L^2([0,T_m]; \Hm{m})$, where $\chi_I$ denotes the characteristic function on $I$. Then for all $\eps'>0$,
\begin{align*}
|(\rho^{(\eps')}-\rho^*, f)_{L^2([0,T_m]; \Hm{m})}| &= |\int_{T_Z} (\rho^{(\eps')} - \rho^*, \rho^*)_{\Hm{m}}\dd t | \\
&= |\int_{T_Z} \left( (\rho^{(\eps')}, \rho^*)_{\Hm{m}} - ||\rho^*||_{\Hm{m}}^2 \right) \dd t | \\
&\geq |\int_{T_Z} -\frac{||\rho^*||_{\Hm{m}}^2}{2} \dd t | \\
&\geq |T_Z| Z^2/2.
\end{align*}
The second last inequality stems from $(\rho^{(\eps')}, \rho^*)_{\Hm{m}} \leq ||\rho^{(\eps')}||_{\Hm{m}} ||\rho^*||_{\Hm{m}}$ and noting that on $T_Z$, $||\rho^*||_{\Hm{m}} \geq Z = 2 K > 2 ||\rho^{(\eps')}||_{\Hm{m}}.$ Thus, $\rho^{(\eps')}$ does not converge weakly to $\rho^*$, a contradiction.
\end{proof}

\begin{lemma}[Strongly Convergent Subsequence]
\label{lem:strongConverge}
Let $\rheps$ be the subsequence obtained from Lemma \ref{lem:weakConverge} and $d/2 < m' < m$.
Then there is a subsequence of $\rheps$ converging strongly to \\ $\rho^* \in \mathcal{C}^0( [0,T_m]; \Hm{m'}).$
\end{lemma}

\begin{proof}
We note that $\Hm{m}$ is compactly embedded in $\Hm{m'}$, which is continuously embedded in $\Lp{\infty}$ (Lemma \ref{lem:sobin}). We also have that $\rheps$ is bounded in $L^\infty([0,T_m]; \Hm{m})$ by Proposition \ref{prop:Ebound}. Thus, if we can show that $\rheps_t$ is bounded in $L^\infty([0,T_m]; \Lp{\infty})$ then, by Aubin-Lions-Dubinski{\u\i} (Lemma \ref{thm:aubin}), there is a strongly convergent subsequence in $L^\infty([0,T]; \Hm{m'}).$ From the proof of Proposition \ref{prop:red},
\begin{align*}
||\rheps_t||_{\Lp{\infty}} &\lesssim 1 + ||\rheps||_{\Hm{m}}^{m}.
\end{align*}
With this, we can conclude there is a strongly convergent subsequence of $\rheps$ converging to $\rho^* \in L^\infty([0,T_m]; \Hm{m'})$ as $\eps \downarrow 0.$ Since each term in the sequence is continuous over $[0,T_m]$ and the convergence is in the $L^\infty$-norm over $[0,T_m]$, by the Uniform Limit Theorem, the limit is also continuous and hence in $\rho^* \in \mathcal{C}^0([0,T_m]; \Hm{m'}).$
\end{proof}

\subsection{Step 4: Obtaining a Strong Solution in Target Space}

To show that the limit $\rho^* \in \mathcal{C}^0([0,T_m]; \Hm{m})$, we work with the weak topology. The key insight for this stage is that continuity in the norm plus weak continuity gives strong continuity.

\begin{lemma}[Weak Continuity]
\label{lem:weak_cont}
Let $\rho_0 \in \Hm{m}$ with integers $m>m'$ with $m>d^* = 3+d$, $m' > d/2$ and 
%an integer $m>m'>d^*:=3+d$ and 
let $\rheps$ and $\rho^*$ be the subsequence and limit found in Lemma \ref{lem:strongConverge}. Then $\forall \eps>0, \rheps \in \mathcal{C}_W([0,T_m]; \Hm{m})$ and $\rheps \to \rho^*$ in $\mathcal{C}_W([0,T_m]; \Hm{m})$ as $\eps \downarrow 0$.
\end{lemma}

\begin{proof}
Since $\rheps \in \mathcal{C}^0([0, T_m]; \Hm{m})$, we also have $\rheps \in \mathcal{C}_W([0, T_m], \Hm{m}).$ Additionally, from $\rheps \rightarrow \rho^*$ uniformly on $\mathcal{C}^0([0,T_m]; \Hm{m'})$, we have that for any $\tilde \phi \in \Hm{-m'}$, $\int_{\dom} \tilde \phi ( \rheps - \rho^*) \dd x \rightarrow 0$ as $\eps \downarrow 0.$

Let $\max\{ ||\rho^*||_{L^\infty([0, T_m]; \Hm{m})}, \sup_{\eps>0} ||\rheps||_{L^\infty([0,T_m]; \Hm{m})} \} \leq U.$

We next note that for $m'<m$, $\Hm{-m'}$ is a dense subset of $\Hm{-m}$ so that given any $\phi \in \Hm{-m}$ and $\nu > 0$, there is $\tilde \phi \in \Hm{-m'}$ with $||\phi - \tilde \phi||_{\Hm{-m}} < \nu/(3U).$

We wish to prove that for any $\phi \in \Hm{-m},$ $\int_{\dom} \phi ( \rheps - \rho^*) \dd x \rightarrow 0$ as $\eps \downarrow 0.$ Let $\nu>0$ and choose $\tilde \phi \in \Hm{-m'}$ so $||\phi - \tilde \phi||_{\Hm{-m}} < \nu/(3U).$ Choose $\eps'$ so that $\eps < \eps'$ gives $|\int_{\dom} \tilde \phi ( \rheps - \rho^*) \dd x| < \nu/3.$ Then if $\eps < \eps'$,
\begin{align*}
|\int_{\dom} \phi ( \rheps - \rho^*) \dd x| &\leq |\int_{\dom} \underbrace{(\phi - \tilde \phi)}_{\in \Hm{-m}} \underbrace{\rheps}_{\in \Hm{m}}| \dd x \\
& +| \int_{\dom} \underbrace{\tilde \phi}_{\in \Hm{-m'}} \underbrace{(\rheps - \rho^*)}_{\text{\quad weak } \Hm{m'} \text{ convergence}} \dd x| \\ &+ \int_{\dom} \underbrace{(\tilde \phi - \phi)}_{\in \Hm{-m}} \underbrace{\rho^*}_{\in \Hm{m}} \dd x| \\
&\leq ||\phi - \tilde \phi||_{\Hm{-m}} ||\rheps||_{\Hm{m}} + |\int_{\dom} \tilde \phi ( \rheps - \rho^*) \dd x | \\ &+ ||\phi - \tilde \phi||_{\Hm{-m}} ||\rho^*||_{\Hm{m}} \\
&\leq \frac{\nu}{3U} U + \frac{\nu}{3} + \frac{\nu}{3U}U \\
&= \nu.
\end{align*}
\end{proof}

\begin{lemma}[Norm Continuity Plus Weak Continuity Gives Strong Continuity]
\label{lem:norm_weak}
Suppose $f \in \mathcal{C}_W([0,T]; \Hm{m})$ and $||f||_{\Hm{m}} \in \mathcal{C}^0([0,T]).$ Then $f \in \mathcal{C}^0([0,T]; \Hm{m}).$
\end{lemma}

\begin{proof}
Let $t_0 \in [0,T]$. We wish to show $\lim_{t \rightarrow t_0} f(t) = f(t_0)$ where the limit is understood to be a left-/right-sided limit at the right-/left-endpoint. We compute
\begin{align*}
\lim_{t \rightarrow t_0} ||f(t) - f(t_0)||_{\Hm{m}}^2 &= \lim_{t \rightarrow t_0} (f(t) - f(t_0), f(t) - f(t_0))_{\Hm{m}} \\
&= \lim_{t \rightarrow 0} \Big( ||f(t)||_{\Hm{m}}^2 + ||f(t_0)||_{\Hm{m}}^2 \\ &- (f(t), f(t_0))_{\Hm{m}} - (f(t_0), f(t))_{\Hm{m}} \Big) \\
&= \lim_{t \rightarrow 0} \Big( ||f(t_0)||_{\Hm{m}}^2 + ||f(t_0)||_{\Hm{m}}^2 \\&- (f(t), f(t_0))_{\Hm{m}} - (f(t_0), f(t))_{\Hm{m}} \Big) \\
&= \lim_{t \rightarrow 0} \Big( (f(t_0), f(t_0) - f(t))_{\Hm{m}} \\ &+ (f(t_0) - f(t), f(t_0))_{\Hm{m}} \Big) \\
&= 0.
\end{align*}
To reach equality 3, we used the continuity of $||f||_{\Hm{m}}.$ To reach equality 5, we used the weak continuity and that $f(t)$ converges weakly to $f(t_0)$ by testing against $f(t_0) \in \Hm{m}.$ 
\end{proof}

To obtain our strong solution result, we first prove a proposition asserting that solutions stemming from initial data $\rho_0 \in \Hm{m}$ in fact stay in $\Hm{m}$ over their interval of existence.

\begin{proposition}
\label{prop:HmHm}
Let $\rho_0 \in \Hm{m}$ with $m>d^*:3+d$ an integer and let $\rheps$ be the subsquence of Lemma \ref{lem:strongConverge} and $\rho^*$ and $T_m$ be as in Lemma \ref{lem:strongConverge}. Then $\rho^* \in \mathcal{C}^0([0,T_m]; \Hm{m}).$
\end{proposition}

To prove Proposition \ref{prop:HmHm}, we prove a lemma affirming that $||\rho^*||_{\Hm{m}}$ is continous at $t=0$. % and that $\rho^* \in \mathcal{C}^0((0,T]; \Hm{m}).$

\begin{lemma}[Continuity Properties of Solution] 
\label{lem:sol_cts}
Let $\rho_0 \in \Hm{m}$ with $m>d^*:=3+d$ an integer and let $\rheps$ and $\rho^*$ be defined as in Lemma \ref{lem:strongConverge}. Then
$||\rho^*||_{\Hm{m}}$ is strongly right-continuous at $t=0$.
\end{lemma}

\begin{proof}
From Eq. \eqref{eq:Eintbd}, for all $\eps>0, \frac{1}{2} ||\rheps||_{\Hm{m}}^2 \leq E_m(t)$ with $E_m(0) = \frac{1}{2} ||\rho_0||_{\Hm{m}}^2$. Then $$\overline{\lim}_{t \downarrow 0} ||\rho^*(t)||_{\Hm{m}} \leq \overline{\lim}_{t \downarrow 0} \overline{\lim}_{\eps \downarrow 0} ||\rheps(t)||_{\Hm{m}} \leq \overline{\lim}_{t \downarrow 0} \sqrt{2 E_m(t)} = ||\rho_0||_{\Hm{m}}.$$

Also, since $\rho^* \in \mathcal{C}_W([0,T_m]; \Hm{m})$, we have $\underline{\lim}_{t \downarrow 0} ||\rho^*(t)||_{\Hm{m}} \geq ||\rho_0||_{\Hm{m}}.$ To see this, note that with $\rho_0 \in \Hm{m}$, 
\begin{align*} ||\rho_0||_{\Hm{m}}^2 &= \lim_{t \downarrow 0} (\rho^*(t), \rho_0)_{\Hm{m}} = \underline{\lim}_{t \downarrow 0} (\rho^*(t), \rho_0)_{\Hm{m}} \\
&\leq \underline{\lim}_{t \downarrow 0} ( ||\rho^*(t)||_{\Hm{m}} ||\rho_0||_{\Hm{m}} )
= ||\rho_0||_{\Hm{m}} \underline{\lim}_{t \downarrow 0} ||\rho^*(t)||_{\Hm{m}}.\end{align*} %In general this would be true as norms are lower semi-continuous. 
Together these give $$\lim_{t \downarrow 0} ||\rho^*||_{\Hm{m}} = ||\rho_0||_{\Hm{m}} = ||\rho^*(0)||_{\Hm{m}}.$$
\end{proof}

Now, we complete the proof of Proposition \ref{prop:HmHm}.

\begin{proof}[Proof of Proposition \ref{prop:HmHm}]
Let $0 < \theta < T_m.$ Then $\rheps \in \mathcal{C}^0((\theta,T_m]; \Hm{m'})$ for all $m'<m+1$ and, in particular, $\rheps$ has a subsequence (as per Lemma \ref{lem:strongConverge}) that converges in $\mathcal{C}^0((\theta,T_m]; \Hm{m})$ and thus we have $\rho^* \in \mathcal{C}^0((0, T_m]; \Hm{m})$ and $||\rho^*||_{\Hm{m}}$ is continuous on $(0, T_m].$ From Lemma \ref{lem:sol_cts}, $||\rho^*||_{\Hm{m}}$ is continuous at $t=0$ so $||\rho^*||_{\Hm{m}} \in \mathcal{C}^0([0, T_m])$. Combined with the weak continuity, we have $\rho^* \in \mathcal{C}^0([0, T_m]; \Hm{m}).$
\end{proof} 

\subsection{Step 5: Convergence Analysis}

In this final step, we finish proving our main result. The most important element is that the mollified equation, with its derivatives, in fact converges uniformly to an equation for $\rho^*$ consistent with \sysaut. Before proving Theorem \ref{thm:local_exist}, we first prove establish a couple of helpful results:

\begin{lemma}[Limiting Derivatives]
\label{lem:limder}
Let $m> d^* := 3+d$ be an integer and suppose that $\rheps$ is the subsequence of approximators converging to $\rho^*$ as in Lemma \ref{lem:strongConverge}. Then for all multi-indices $|\alpha| \leq 2$, $\lim_{\eps \downarrow 0} \Jeps [\partial^\alpha \rho^{(\epsilon)}] = \partial^\alpha \rho^*$ uniformly on $\dom$ for a.e. $t \in [0,T_m].$
\end{lemma}

\begin{proof}
It is done by computation and bounding an error. We use Lemmas \ref{lem:moll} and \ref{lem:sobin}.
\begin{align*}
\Jeps [\partial^\alpha \rho^{(\epsilon)}] &= \Jeps [\partial^\alpha \rho^*] + \Jeps [\partial^\alpha (\rho^{(\epsilon)} - \rho^*)].
\end{align*}
From $\rho^* \in \mathcal{C}^0([0,T_m]; \Hm{m})$, we have $D^2 \rho^*(t) \in \mathcal{C}^{\ell}(\dom)$ where $\ell>d/2$. Thus, for $0 \leq |\alpha| \leq 2,$ $\lim_{\eps \downarrow 0} \Jeps [\partial^\alpha \rho^*(t)] = \partial^\alpha \rho^*(t)$ uniformly.

Let $m>m'>2+d/2$. We have $\rheps \to \rho^*$ in $\Hm{m'}$ on $[0,T_m]$ with $\partial^\alpha (\rheps - \rho^*)$ again $\in \mathcal{C}^{\ell}$, $\ell>d/2$. Then with Lemma \ref{lem:moll} item 1,
\begin{align*}
||\Jeps [\partial^\alpha (\rho^{(\epsilon)} - \rho^*)]||_{\Lp{\infty}} &\leq ||\partial^\alpha ( \rheps - \rho^*) ||_{\Lp{\infty}} \\
& \lesssim ||\rheps - \rho^*||_{\Hm{m'}} \to 0.
\end{align*}
\end{proof}

\begin{proposition}[Uniform Convergence of $\rheps_t$]
\label{prop:uniformC}
Let $m>  d^* := 3+d$ and suppose that $\rheps$ is the subsequence of approximators converging to $\rho^*$ as in Lemma \ref{lem:strongConverge}. Then the sequence $\rheps_t$ converges uniformly in $\mathcal{C}^0([0,T_m]; \mathcal{C}^0(\dom)).$
\end{proposition}

\begin{proof}
We consider groups of terms for $\rheps_t$: the terms involving the Laplacian, the nonlocal operator, and all remaining terms. Denote $u^* = u(\kappa, \rho^*).$ We wish to show that all of the spatial terms and derivatives converge uniformly. We begin with the derivative terms (see Eq. \eqref{eq:lap_rhou}):
\begin{align*}
&\delta \lim_{\eps \downarrow 0} \Big( \Jeps [ \lap( u^+ \Jeps \rho^{(\epsilon)} - \kappa M(\Jeps \rho^{(\epsilon)}) ) ] \Big) \\
&= \delta \lim_{\eps \downarrow 0} \Big( u^+ \Jeps [ \lap \Jeps [\rho^{(\epsilon)}]] - \Jeps[ M(\Jeps \rho^{(\epsilon)}) \lap \kappa + 2 M'(\Jeps \rho^{(\epsilon)}) \nabla \kappa \cdot \nabla \Jeps \rho^{(\epsilon)} \\
& + \kappa M'(\Jeps \rho^{(\epsilon)}) \lap \Jeps \rho^{(\epsilon)} + \kappa M''(\Jeps \rho^{(\epsilon)}) | \nabla \Jeps \rho^{(\epsilon)} |^2 ] \Big) \\
&= \delta \lim_{\eps \downarrow 0} \Big( u^+ \Jeps \lap \rho^{(\epsilon)} - M(\Jeps \rho^{(\epsilon)}) \lap \kappa - 2 M'(\Jeps) \nabla \kappa \cdot \nabla \Jeps \rho^{(\epsilon)} \\
& - \kappa M'(\Jeps \rho^{(\epsilon)}) \lap \Jeps \rho^{(\epsilon)} - \kappa M''(\Jeps \rho^{(\epsilon)}) | \nabla \Jeps \rho^{(\epsilon)} |^2  \Big) \\
&= \delta \lim_{\eps \downarrow 0} \Big( u^+ \lap \rho^{*} - M(\Jeps \rho^{(\epsilon)}) \lap \kappa - 2 M'(\Jeps \rho^{\epsilon}) \nabla \kappa \cdot \nabla \rho^{*} \\
& - \kappa M'(\Jeps \rho^{(\epsilon)}) \lap \rho^{*} - \kappa M''(\Jeps \rho^{(\epsilon)}) | \Jeps \nabla \rho^{(\epsilon)} |^2  \Big) \\
&= \delta \Big( u^+ \lap \rho^* - M(\rho^*) \lap \kappa - 2 M'(\rho^*) \nabla \kappa \cdot \nabla \rho^* - \kappa M'(\rho^*) \lap \rho^* - \kappa M''(\rho^*) |\nabla \rho^*|^2 \Big) \\
&= \delta \lap(\rho^* u^*).
\end{align*}

Above, the second equality came from $\Jeps [f]$ converging uniformly to $f$ when $f \in \mathcal{C}^{\ell}(\dom)$ and $\ell>d/2$. Each expression is in fact in $\mathcal{C}^\infty(\dom)$ due to the smooth functions and further mollifications. 
The third inequality comes from Lemma \ref{lem:limder} and from selecting $2+m' <m$ and noting that 
\begin{align*}
||\lap \Jeps[\rheps] - \lap \rho^*||_{\Lp{\infty}} &\lesssim ||\lap \Jeps[\rheps] - \lap \rho^*||_{\Hm{m'}} \\
&\leq ||\Jeps[\rheps] - \rho^*||_{\Hm{m'+2}} \\
&\leq ||\Jeps[\rheps] - \Jeps[\rho^*]||_{\Hm{m'+2}} + ||\Jeps[\rho^*] - \rho^*||_{\Hm{m'+2}}.
\end{align*}
Both terms go to zero: the first from $\rheps \to \rho^*$ in $\Hm{m'+2}$ (since $d/2 < m'+2 < m$) and the second from $\rho^* \in \Hm{m}$ (item 4 of Lemma \ref{lem:moll}). Similar logic applies to the $\nabla \Jeps[\rheps] \to \nabla \rho^*.$
The fourth equality stems from $\Jeps [\partial^\alpha \rho^{(\epsilon)}]$ converging uniformly to $\partial^\alpha \rho^*$ for $0 \leq |\alpha| \leq 2$ by Lemma \ref{lem:limder} and making use of the Lipschitz properties of $M$, $M'$ and $M''.$

The local terms work out trivially since, for each $t$, $\rho^{(\epsilon)}$ converges uniformly to $\rho^*$:
\begin{align*}
\lim_{\eps \downarrow 0} \Big( \eta - \omega \rho^{(\epsilon)} - \gamma ( u^+ \rho^{(\epsilon)} - \kappa M( \rho^{(\epsilon)} ) \Big) &= \eta - \omega \rho^* - \gamma( u^+ \rho^* - \kappa M(\rho^*) ).
\end{align*}

Finally, the nonlocal operator can be handled by the Dominated Convergence Theorem \cite{rudin1987real}. We have that $|\tau(y,x) \gamma(y) \rho^{(\epsilon)}(t,y) u(\kappa,\rheps) | \lesssim ||\rho^{(\epsilon)}||_\infty$ which is finite on $[0, T_m]$ due to Proposition \ref{prop:red} and Sobolev Embedding. %Proposition \ref{prop:bounda}.
 Thus, for each $x \in \dom$, the integrand of the nonlocal operator is bounded by an $\Lp{1}$ constant function and from the convergence of $\rho^{(\epsilon)}$, we conclude
\begin{equation*}
\lim_{\eps \downarrow 0} \int_{\Td}  \tau(y,x) \gamma(y) \rho^{(\epsilon)}(t,y) u(\kappa(y),\rheps(t,y)) \dd y = \int_{\Td} \tau(y,x) \gamma(y) \rho^*(t,y) u(\kappa(y) \rho^*(t,y)) \dd y.
\end{equation*}

To see that the convergence is uniform, let $$G^{(\eps)}(x) = \int_\dom |\tau(y,x) \gamma(y) ( \rho^{(\epsilon)}(t,y) u(\kappa(y),\rheps(t,y)) - \rho^{*}(t,y) u^{*}(\kappa(y),\rho^*(t,y)) )| \dd y.$$ Then with $|\tau|, |\gamma| \leq \Xi_0$, we have $$|G^{(\eps)}(x)| \leq \Xi_0^2 \int_\dom | \rho^{(\epsilon)}(t,y) u(\kappa(y),\rheps(t,y)) - \rho^{*}(t,y) u(\kappa(y), \rho^*(t,y))  | \dd y \to 0 \quad \text{as } \eps \downarrow 0$$ since $\rheps \to \rho^*$ uniformly and $|\dom| = 1.$

\end{proof}

\begin{proof}[Proof of Theorem \ref{thm:local_exist}]
We have $m > d^* := 3+d$ so that $\Hm{m} \subset \mathcal{C}^2(\dom)$ and $\Hm{m-2} \subset \mathcal{C}^0(\dom)$ from Lemma \ref{lem:sobin}. From Proposition \ref{prop:HmHm}, we have the existence of a sequence $\rho^{(\epsilon)} \rightarrow \rho^* \in \mathcal{C}^0([0, T_m]; \Hm{m})$. 

We first verify that $\rho^*$ is a classical solution. Since $\rho^* \in \mathcal{C}^0([0, T_m]; \Hm{m})$, it is also in $\mathcal{C}^0([0, T_m]; \mathcal{C}^2(\dom)).$ From Proposition \ref{prop:uniformC}, the sequence $\rheps_t$ converges uniformly over $\mathcal{C}^0([0,T_m]; \mathcal{C}^0(\dom)).$ Additionally, $\lim_{\eps \downarrow 0} \rheps(0) = \rho_0$ (there is pointwise convergence at $t=0$). Therefore $\lim_{\eps \downarrow 0} \rheps_t = \partial_t \lim_{\eps \downarrow 0} \rheps = \rho^*_t$ and we conclude $\rho^*(t,x) : [0,T_m] \times \dom \to \nnR$ satisfies $$\rho^*_t = \delta \lap (\rho^* u^*) + \eta - \omega \rho^* + I[\gamma \rho^* u^*] - \gamma \rho^* u^*.$$ Since $\rho^*_t \in \mathcal{C}^0(\dom)$, we also have $\rho^* \in \mathcal{C}^1([0,T_m]; \mathcal{C}^0(\dom)).$

Upon examination, $\rho^*_t \in \Hm{m-2}$ since $\rho^* \in \Hm{m}$ with $D^2 \rho^* \in \Hm{m-2}$ on $[0,T_m]$, all the data are smooth, and $m-2>d/2$. Thus $\rho^* \in \mathcal{C}^0([0,T_m]; \Hm{m}) \cap \mathcal{C}^1([0,T_m]; \Hm{m-2}).$

\end{proof}

\begin{remark}
From Lemma \ref{lem:strongConverge} we could have already established a classical solution since $m>m'>d^*$ is sufficient for the uniform convergence of $\rheps_t.$
\end{remark}

\section{Conclusions and Future Work}

\label{sec:conclusion}

We have proven the local existence of classical solutions to a particular nonlocal, nonlinear parabolic PDE with time-independent data. We did so through parameterizing a family of regularized solutions with smoothing parameter $\eps$ and establishing an interval of local existence for the family independent of $\eps$. By passing to a subsequence, we obtained strong convergence to a solution of the PDE. This methodology is quite broadly applicable to nonlinear PDEs, but from these steps, we did not obtain global existence. However, given the smoothing properties of parabolic PDEs, the nondegeneracy condition in $u$, and the fact the limiting behaviours of the diffusion term are well-behaved as $\rho \downarrow 0$ and $\rho \uparrow \infty$, we speculate that a stronger result holds: that there exists a unique solution, smooth in all its arguments, that exists globally in time.

The original model allowed for time-dependent data. However, Picard's Theorem does not generalize well to nonautonomous ODEs in Banach spaces. While we can establish local existence of regularized solutions, the continuation result does not hold. This prevented our obtaining a limiting solution as $\eps \downarrow 0$ with time-dependent data. It would be of interest to find other means of establishing existence of solutions with time-dependent data. One perspective is that for many applications, piecewise constant data may be useful. Thus, assuming the solutions do exist long enough, piecewise constant data can be combined sequentially.

Other useful future directions include studying the PDE with initial conditions in a fractional Sobolev space (primarily of mathematical interest). Our proofs in bounding the energy of Step 2, for example, were done for integer values of $m$. As an additional consideration, the model could be studied with fractional diffusion. Buildings and roads in a city effectively make for a porous medium and we take note that fractional diffusion equations \cite{erochenkova2001fractional} can be used in such settings to more accurately model dispersion.

\section*{Acknowledgments}
The author thanks Andrea Bertozzi for general guidance in the analysis of nonlinear PDEs, and for her illuminating discussions and elaborations on the more technical analytic tools needed in this work (especially in Step 4). The author also thanks Terence Tao and Inwon Kim for taking the time to discuss this work and providing valuable suggestions and the reviewers for their detailed constructive feedback in revising this manuscript.

The author acknowledges funding from the Natural Sciences and Engineering Research Council of Canada (NSERC) [PDF-502-479-2017]. Cette recherche a \'et\'e financ\'ee par le Conseil de recherches en sciences naturelles et en g\'enie du Canada (CRSNG) [PDF-502-479-2017].

\appendix

\section{Background and Further Proofs}

\label{app:nom}

\subsection{Background}

\paragraph{Derivatives:} We will use both $D^m$ and $\partial^\alpha$ to denote spatial derivatives in $x$. For an integer $m \geq 0$, we denote $$D^m f$$ to be {\it all $m^{\text{th}}$-order $x$-derivatives of $f$}. So writing $|D^2 f| \leq C$ we mean that all second-order spatial derivatives of $f$ are bounded by $C$. For a multi-index $\alpha$, we denote $$\partial^\alpha f$$ to be the {\it single $|\alpha|^{\text{th}}$-order $x$-derivative of $f$ specified by $\alpha$.}

\paragraph{Hilbert Space $H^m(\Omega)$:} for a domain $\Omega$, $m \in \mathbb{N} \cup \{0\}$, and $f,g : \Omega \to \mathbb{R}$, we define an inner product $$( f, g )_{H^m(\Omega)} := \sum_{0 \leq |\alpha| \leq m} \int_\Omega \partial^\alpha f \partial^\alpha g \dd x$$ with induced norm $$||f||_{H^m(\Omega)} = ( f, f )_{H^m(\Omega)}^{1/2}.$$ Then we define $$H^m(\Omega) = \{ f : \Omega \to \mathbb{R} | \quad ||f||_{H^m(\Omega)} < \infty \}.$$

If $\Omega = \Td$ then an equivalent norm upon $\Hm{m}$ for real-valued $m$ can be defined via
$$||f||_{H^m(\Omega)} = \sum_{k \in \mathbb{Z}^d} (1 + |k|^2)^{m/2} |\hat f(k)|^2.$$

\paragraph{Banach Space $L^p(\Omega)$:} for a domain $\Omega$, $1 \leq p \leq \infty$, and $f: \Omega \to \mathbb{R}$, we define the norm 
$$||f||_{L^p(\Omega)} = \begin{cases} \left( \int_{\Omega} |f|^p \dd x \right)^{1/p}, \quad 1 \leq p < \infty \\
\text{ess sup}_\Omega |f|, \quad p = \infty. \end{cases}$$ Then we define $$L^p(\Omega) = \{ f : \Omega \to \mathbb{R} | f \text{ is measurable and } ||f||_{L^p(\Omega)} < \infty \}.$$ Naturally $H^0(\Omega) = L^2(\Omega).$ %{\note{ IS MEASURABLE ALREADY IMPLIED? WHAT ABOUT FOR THE $\infty$ CASE?}}

\paragraph{H\"{o}lder Space $\mathcal{C}^{n,\alpha}(\Omega)$:} for a domain $\Omega$, $n \in \mathbb{N} \cup \{0\}$ and $0 < \alpha \leq 1$ we define the norm $$||f||_{\mathcal{C}^{n,\alpha}(\Omega)} = \sup_{|\beta| \leq n} ||\partial^\beta f||_{L^\infty(\Omega)} + \sup_{\substack{x,y \in \Omega, x \neq y \\ |\beta|=n}} \frac{|f^{(\beta)}(x) - f^{(\beta)}(y)|}{|x-y|^{\alpha}}.$$ Then $$\mathcal{C}^{n,\alpha}(\Omega) = \{ f | \quad ||f||_{\mathcal{C}^{n,\alpha}(\Omega)} < \infty \}.$$

As a convention, we will sometimes use $\mathcal{C}^\ell(\Omega)$ to indicate a H\"older space. For $\ell \in \{0,1,...\}$ the notation represents the number of continuous derivatives. But for fractional $\ell$, let $\ell = \lfloor \ell \rfloor + \{ \ell \}$ decompose $\ell$ into a whole number part and a fractional part. We then define $\mathcal{C}^\ell(\Omega) := \mathcal{C}^{\lfloor \ell \rfloor, \{ \ell \}}(\Omega)$.

\paragraph{Banach Space $L^p( I; X )$:} for an interval $I \subset \mathbb{R}$, let $X$ be a Banach space and $1 \leq p \leq \infty$. For $f : I \to X$, we define the norm $$||f||_{L^p( I; X)} = \left| \left| \Big( || f(t,\cdot) ||_{X} \Big) \right| \right|_{L^p(I)}.$$ Then we define $$L^p(I; X) = \{ f : I \to X | \text{f is Bochner integrable and } ||f||_{L^p(I ; X)} < \infty \}.$$ A reader may refer to a book such as that of Ladas and Lakshmikantham \cite{abstract_de} for a definition of the Bochner integral.

\begin{remark}
If $p=2$ and $X=H$ is a Hilbert space then $L^2(I; H)$ is also a Hilbert space with inner product $(f,g)_{L^2(I; H)} = \int_I (f,g)_H \dd t.$
\end{remark}

\paragraph{Banach Space $\mathcal{C}^n( I; X )$:} for an interval $I \subset \mathbb{R}$, let $X$ be a Banach space and $n \geq 0$ be an integer. We define $$\mathcal{C}^n( I; X ) = \{ f: I \to X | f \text{ is continuously differentiable up to order } n \}.$$ 

For a function $f \in \mathcal{C}^n(I; X)$, we denote $\dot f = f'(t)$ as the derivative of the map $f: I \to X.$

\paragraph{Space of Continuous Functions in the Weak Topology of $H^m(\Omega)$, \\ $\mathcal{C}_W(I; H^m(\Omega))$:} for an interval $I$, we define $$\mathcal{C}_W(I; H^m(\Omega)) = \{ f : I \to \Hm{m} | \forall \phi \in \Hm{-m}, \int_\Omega f(\cdot) \phi \dd x \in \mathcal{C}^0(I) \}.$$ We can equivalently define $$\mathcal{C}_W(I; H^m(\Omega)) = \{ f : I \to \Hm{m} | \forall \Phi \in \Hm{m}, (\Phi, f(\cdot))_{\Hm{m}} \in \mathcal{C}^0(I) \}.$$

\paragraph{Inequalities:} We will write $$f \lesssim g$$ if $f, g \geq 0$ and $f \leq C g$ for some constant $C>0$. In general, the constant $C$ could depend upon model parameters. When this dependence is important, we will write $$f \lesssim_{a,b,c} g$$ to indicate, for instance, if the constant $C$ depends on $a, b,$ $c$, and possibly other parameters that are not relevant. 

In a like fashion, we write $f \sim g$ if $f, g \geq 0$ and $C_1 f \leq g \leq C_2 f$ for some $C_1, C_2 > 0.$

\paragraph{Chain Rule:}

\begin{lemma}[Fa\`a di Bruno's Formula \cite{hardy2006combinatorics}] 	Let $f:\mathbb{R} \to \mathbb{R}$ and $g:\mathbb{R}^n \to \mathbb{R}.$ Let $\alpha$ be a multi-index with order at least 1. Then	
	$$\partial^\alpha f(g(x)) = \sum_{\beta \in \mathcal{P}(\alpha)} f^{(|\beta|)} (g(x)) \prod_{\sigma \in \beta} \partial^\sigma g(x)$$	
	where $\mathcal{P}(\alpha)$ represents all partitions of $\alpha$ allowing for multiplicity.
\end{lemma}

Notationally there are many ways of writing Fa\`a di Bruno's Formula and we have adopted one that is useful for our computations. To clarify any ambiguities in the statement of the equation, we consider a concrete example below. 

Let $f: \mathbb{R} \to \mathbb{R},$ $g: \mathbb{R}^2 \to \mathbb{R}$, and consider computing $\frac{\partial^3}{\partial^2 x_1 \partial x_2} f(g(x)).$ We have $\alpha = (2,1)$ which we identify with the multiset $\{1_1,1_2,2\}$ that tracks the derivatives of each component (the subscipts are added for clarity). This multiset of 3 elements can (allowing for multiplicity) be split into 5 partitions: \begin{align*}
&\{1_1,1_2,2\} \qquad \text{i.e. } \partial_{x_1 x_1 x_2} g \\
&\{1_1,1_2\}, \{2\} \qquad \text{i.e. } \partial_{x_1 x_1} g \times \partial_{x_2} g \\
&\{1_1\}, \{1_2,2\} \qquad \text{i.e. } \partial_{x_1} g \times \partial_{x_1 x_2} g  \\
&\{1_2\}, \{1_1,2\} \qquad \text{i.e. } \partial_{x_1} g \times \partial_{x_1 x_2} g  \\
&\{1_1\}, \{1_2\}, \{2\} \qquad \text{i.e. } \partial_{x_1} g \times \partial_{x_1} g \times \partial_{x_2} g 
\end{align*}

Thus, the set \begin{align*} \mathcal{P}(\alpha) &= \{ \{ (2,1) \}, \\
& \{ (2,0), (0,1) \}, \\
& \{ (1,0), (1,1) \}, \\
& \{ (1,0), (1,1) \}, \\
& \{ (1,0), (1,0), (0,1) \}  \}. \end{align*} The first element of $\mathcal{P}(\alpha)$ corresponds to $|\beta|=1$, the next three elements correspond to $|\beta|=2$, and the last element corresponds to $|\beta|=3.$ Summing all of this up, keeping each term distinct:

\begin{align*} \frac{\partial^3}{\partial^2 x_1 \partial x_2} f(g(x)) &= f'(g(x)) \frac{\partial^3 g}{\partial^2 x_1 \partial x_2} + f''(g(x)) \frac{\partial^2 g}{\partial x_1^2} \frac{\partial g}{\partial x_2} + f''(g(x)) \frac{\partial g}{\partial x_1} \frac{\partial^2 g}{\partial x_1 \partial x_2} \\ &+ f''(g(x)) \frac{\partial g}{\partial x_1} \frac{\partial^2 g}{\partial x_1 \partial x_2} + f'''(g(x)) \frac{\partial g}{\partial x_1} \frac{\partial g}{\partial x_1} \frac{\partial g}{\partial x_2}. \end{align*}

\paragraph{Nonautonomous PDEs}

We require the version of Picard-Lindeloff in a Banach space below to establish local existence of the approximators to the nonautonomous system. See Theorem 5.1.1 of \cite{abstract_de}. 

\begin{lemma}[Picard-Lindeloff for Nonautonomous ODE in Banach Space]
	\label{thm:picard}	
	Let $(X,||\cdot||_X)$ be a Banach space with $v_0 \in X$ and define the rectangle $\mathcal{R}_0 = \{(t,v) \in \mathbb{R} \times X \mid |t-t_0| \leq \alpha, ||v-v_0||_X \leq \beta \}.$ Let $f: \mathcal{R}_0 \to X$ be continuous in $t$ for each fixed $v$. Assume that $||f(t,v)|| \leq M$ on $\mathcal{R}_0$ (uniformly bounded) and that for $(t,v_1), (t,v_2) \in \mathcal{R}_0$ we have $||f(t,v_1) - f(t,v_2)||_X \leq K ||v_1 - v_2||_X$ (uniformly Lipschitz in $t$) for nonnegative constants $M$ and $K$. Let $\alpha$ and $\beta$ be positive constants such that $\alpha M \leq \beta$. Then there exists a unique strongly continuously differentiable function $v(t)$ satisfying \begin{align*}
	v'(t) &= f(t,v(t)), \quad |t - t_0| \leq \alpha \\
	v(t_0) &= v_0.
	\end{align*}
\end{lemma}

For a nonautonomous ODE, it is very difficult to say more, such as extending solutions \cite{abstract_de}. 

\subsection{Ancillary Results and Proofs}

Minor results used to establish the main results of this manuscript are included here. We begin with Corollaries to Lemma \ref{lem:sobin}.

\begin{corollary}
	\label{cor:hjm}
	Let $\Omega \subset \mathbb{R}^d$. Let $m > d/2$ be an integer and $v \in H^m(\Omega)$. Then for all $j \leq m$, $$||D^j v||_{L^{2m/j}(\Omega)} \lesssim ||v||_{H^m(\Omega)} < \infty.$$
\end{corollary}

\begin{proof}%[Proof of Corollary \ref{cor:hjm}]
	Choose $q=\infty$, $r=2,$ and $\alpha = j/m$ in the Gagliardo-Nirenberg interpolation inequality of Lemma \ref{lem:sobin} yielding $p = 2m/j$ so that \begin{align*} ||D^j v||_{L^{2m/j}(\Omega)} &\lesssim ||v||_{L^{\infty}(\Omega)}^{1-j/m} ||D^m v||_{L^{2}(\Omega)}^{j/m} \\
	&\leq ||v||_{L^{\infty}(\Omega)}^{1-j/m} ||v||_{H^{m}(\Omega)}^{j/m} \\
	&\lesssim ||v||_{H^m(\Omega)}. \end{align*} 
	
	In arriving at the last line we used that $m>d/2$ so that $||v||_{L^{\infty}(\Omega)} \lesssim ||v||_{H^{m}(\Omega)}.$
\end{proof}

\begin{corollary}
	\label{cor:okprod}
	Let $m>d/2$ be an integer and $v \in \Hm{m}$. Suppose $S$ is a set of multi-indices such that $\sum_{s \in S} |s| \leq 2m$. Then $$\left| \int_{\dom} \prod_{s \in S}  \partial^s v \dd x \right| \lesssim ||v||_{\Hm{m}}^{|S|} < \infty.$$ Furthermore, if $\max_{s \in S} \{ |s| \} \leq m$ then $$\int_{\dom} | \prod_{s \in S} \partial^s v | \dd x \lesssim ||v||_{\Hm{m}}^{|S|}.$$
\end{corollary}

\begin{proof}
	First assume that $\max_S |s| \leq m$. By corollary \ref{cor:hjm}, for $s \in S$, $\partial^s v \in \Lp{2m/|s|}.$ Also, $$\sum_{s \in S} \frac{1}{2m/|s|} := \frac{1}{r} \leq 1.$$ We will use the the generalized H\"{o}lder's inequality in Lemma \ref{lem:sobin} (item 3). We have %Since for a finite measure space $||\cdot||_{\Lp{q'}} \lesssim ||\cdot||_{\Lp{q''}}$ for $q'' \geq q'$,
	\begin{align*}
	|\int_{\dom} \prod_{s \in S} \partial^s v \dd x| &\leq \int_{\dom} |\prod_{s \in S} \partial^s v| \dd x \\
	&\lesssim ||\prod_{s \in S} \partial^s v||_{\Lp{r}} \\ &\leq \prod_{s \in S} ||\partial^s v||_{\Lp{2m/|s|}} \\ &\lesssim \prod_{s \in S} ||v||_{\Hm{m}} \\
	&= ||v||_{\Hm{m}}^{|S|}.
	\end{align*}
	The second line comes from noting that $r \geq 1$ so that the $\Lp{1}$-norm in the first line is controlled by the $\Lp{r}$-norm. The third line comes from the Generalized H\"{o}lder's inequality and the fourth line comes from Corollary \ref{cor:hjm}.
 
If there is $s \in S$ such that $|s| = m + N$ for some $N>0$ then the sum of all remaining derivative orders cannot exceed $m-N$. By applying integration by parts $N$ times, without boundary terms, we can arrange a sum of terms on which the largest order derivative is $m$ and the sum of all remaining derivative orders is no larger than $m$. Then the preceding argument holds.
\end{proof}

\begin{lemma}[Expanding Products in Differences]
	\label{lem:expand}	
Let $S$ be a finite index set of size $|S|$ and place an (arbitrary) strong ordering $\prec$ on its elements so $S = \{ s_1, s_2, ..., s_{|S|} \}$ where $s_i \prec s_{i+1}$ for $i=1, ..., |S|-1.$ Suppose that we enumerate $\{f_s\}_{s \in S}$ and $\{g_s\}_{s \in S}$ in such a manner. Then $$\prod_{i=1}^{|S|} f_{s_i} - \prod_{i=1}^{|S|} g_{s_i} = \sum_{i=1}^{|S|} (f_{s_i} - g_{s_i}) \left( \prod_{1\leq j<i} f_{s_j} \prod_{|S|\geq k > i} g_{s_k}  \right) .$$	
\end{lemma}

\begin{proof}
	We proceed by induction on $|S|$ and assume without loss of generality that $S=\{1,2,...,n\}.$ If $|S| = 1$ then we only have $f_1 - g_1$ so that both products are empty and the sum is just $f_1 - g_1.$ Otherwise, suppose the identity holds for $|S| = n$. Then
	\begin{align*}
	f_1 ... f_{n+1} - g_1 ... g_{n+1} &= f_1 ... f_n (f_{n+1} - g_{n+1}) + \big( f_1 ... f_n - g_1 ... g_n \big) g_{n+1} \\ &=  f_1 ... f_n (f_{n+1} - g_{n+1}) + \sum_{i=1}^{n} \left( (f_i - g_i) \prod_{1\leq j<i} f_j \prod_{n \geq k>i} g_k  \right) g_{n+1} \\
&= (f_{n+1}-g_{n+1}) \prod_{1 \leq j < n+1} f_j \overbrace{\prod_{n+1 \geq k > n+1} g_k}^{=1} \\ &+ \sum_{i=1}^n \left( (f_i - g_i) \prod_{1 \leq j<i} f_j \prod_{n+1 \geq k>i} g_k  \right)  \\ 
 &= \sum_{i=1}^{n+1} \left( (f_i - g_i) \prod_{1 \leq j<i} f_j \prod_{n+1 \geq k>i} g_k \right).
	\end{align*}
\end{proof}

\begin{lemma}
	\label{lem:prod}
Let $\Omega \subset \mathbb{R}^d$ be bounded. Let $f \in \mathcal{C}^m(\Omega)$ and $g \in H^m(\Omega)$ for an integer $m\geq0$. Then $||fg||_{H^m(\Omega)} \lesssim % C\{f,\Omega\} 
||g||_{H^m(\Omega)}.$
\end{lemma}

\begin{remark}
	If $m>d/2$ then since $f \in \mathcal{C}^m(\Omega)$, we have $f \in H^m(\Omega)$, too. The Banach algebra property for $m>d/2$ makes the proof trivial.
\end{remark}

\begin{proof}[Proof of Lemma \ref{lem:prod}]
	\begin{align*}
	||fg||_{H^m(\Omega)}^2 &= \sum_{|\alpha|\leq m} \int_{\Omega} |\partial^\alpha (fg)|^2 \dd x \\ &= \sum_{|\alpha|\leq m} \int_{\Omega} \left| \sum_{\beta \leq \alpha} {\alpha \choose \beta} \partial^\beta f \partial^{\alpha-\beta} g \right| ^2 \dd x \\
	&\leq \sum_{|\alpha| \leq m} |\alpha|!^2 (\sup_{\substack{|\beta| \leq |\alpha| \\ x \in \Omega}} |\partial^\beta f|^2) \int_\Omega \sum_{\beta \leq \alpha} |\partial^{\alpha-\beta} g|^2 \dd x \\
	&\lesssim_{f,m} ||g||_{H^m(\Omega)}^2.
	\end{align*}
\end{proof}

\begin{proposition}
	\label{prop:chain}
	Let $f: \mathcal{D} \subset \mathbb{R} \to \mathbb{R}$ and $g_1, g_2: \Omega \subset \mathbb{R}^{d} \to \mathcal{D}$ be such that $f \in \mathcal{C}^m(\mathcal{D})$ and $g_1, g_2 \in H^m(\Omega)$ for an integer $m > d/2.$ If $f(0) = 0$ and $f$ has bounded derivatives on $\mathcal{D}$ up to order $m$ then \begin{multline*} ||f \circ g_1 - f \circ g_2||_{H^m(\Omega)} \lesssim_{m, ||g_1||_{H^m(\Omega)}} ||g_1-g_2||_{H^m(\Omega)} + \mathcal{O}( ||g_1-g_2||_{H^m(\Omega)}^2  ), \quad ( ||g_1-g_2||_{H^m(\Omega)} \downarrow 0).\end{multline*}  
\end{proposition}

\begin{remark}
	\label{rmk:fog}
	In other words $f \circ g$ is locally Lipschitz in $H^m(\Omega)$ when acting upon functions of $H^m(\Omega).$ We also note that in examining the proof, that the corresponding Lipschitz constant can be expressed as a continuous function of $||g_1||_{H^m(\Omega)}$ where the constant $\lesssim_m (1+||g_1||_{H^m(\Omega)}^{m}).$
\end{remark} 

\begin{proof}[Proof of Proposition \ref{prop:chain}]
	For convenience we denote $K_m$ to be the bound on the absolute value of the $m^{\text{th}}$ derivative of $f$.
	
	We need to compute $$\sum_{0 \leq |\alpha| \leq m} \int_\Omega |\partial^\alpha (f \circ g_1 - f \circ g_2)|^2 \dd x.$$
	Note that when $\alpha = 0$ we have $|\partial^\alpha (f \circ g_1 - f \circ g_2)| = |f \circ g_1 - f \circ g_2 | \leq K_1 |g_1 - g_2|$. Thus, $\int_\Omega |\partial^\alpha (f \circ g_1 - f \circ g_2)|^2 \dd x \leq K_1^2 ||g_1 - g_2||_{H^0(\Omega)}^2 \leq K_1^2 ||g_1 - g_2||_{H^m(\Omega)}^2.$	
	For $|\alpha| > 0$, we begin by examining \begin{align} |\partial^\alpha (f \circ g_1 - f \circ g_2)|^2 &= \left( \sum_{\beta \in \mathcal{P}(\alpha)} \left( f^{(|\beta|)} (g_1) \prod_{\sigma \in \beta} \partial^\sigma g_1 - f^{(|\beta|)} (g_2) \prod_{\sigma \in \beta} \partial^\sigma g_2 \right) \right)^2 \nonumber \\ &= \sum_{\beta, \beta' \in \mathcal{P}(\alpha)} \left( f^{(|\beta|)} (g_1) \prod_{\sigma \in \beta} \partial^\sigma g_1 - f^{(|\beta|)} (g_2) \prod_{\sigma \in \beta} \partial^\sigma g_2 \right) \times \nonumber \\& \left( f^{(|\beta'|)} (g_1) \prod_{\sigma' \in \beta'} \partial^{\sigma'} g_1 - f^{(|\beta'|)} (g_2) \prod_{\sigma' \in \beta'} \partial^{\sigma'} g_2 \right). \label{eq:fogd2}
	\end{align} 
	
	Thus, the squared $\alpha$-th derivative is a sum of products of two factors of the form $$\left( f^{(|\beta|)} (g_1) \prod_{\sigma \in \beta} \partial^\sigma g_1 - f^{(|\beta|)} (g_2) \prod_{\sigma \in \beta} \partial^\sigma g_2 \right).$$ We now attempt to prove each of such factors are in $L^2(\Omega)$ so that Cauchy-Schwarz can be utilized. Assuming this can be done, \begin{align*} 
&|| f^{(|\beta|)} (g_1) \prod_{\sigma \in \beta} \partial^\sigma g_1 - f^{(|\beta|)} (g_2) \prod_{\sigma \in \beta} \partial^\sigma g_2 ||_{L^2(\Omega)} \\ 
&\leq \overbrace{|| f^{(|\beta|)} (g_1) \left( \prod_{\sigma \in \beta} \partial^\sigma g_1 - \prod_{\sigma \in \beta} \partial^\sigma g_2 \right) ||_{L^2(\Omega)}}^{T_1} \\ &+ \overbrace{|| ( f^{(|\beta|)} (g_1) - f^{(|\beta|)} (g_2) ) \prod_{\sigma \in \beta} \partial^\sigma g_2||_{L^2(\Omega)}}^{T_2}. \end{align*}
	
	Now we must justify that both terms on the right-hand side are in fact in $L^2(\Omega)$ and, more importantly, with a nice dependence upon $||g_1 - g_2||_{H^m(\Omega)}.$ We commence with $T_2$ noting that from Corollary \ref{cor:hjm}, each $\partial^\sigma g_2 \in L^{2|\alpha|/|\sigma|}(\Omega)$ and, furthermore, $\sum_{\sigma \in \beta} |\sigma| = |\alpha| \leq m.$ Thus, by the generalized H\"{o}lder's inequality with $\sum_{\sigma \in \beta} \frac{1}{2|\alpha|/|\sigma|} = \frac{1}{2}$, the product must be in $L^2(\Omega).$ As the difference term in $T_2$ is bounded by $K_{|\beta|} ||g_1-g_2||_{L^\infty(\Omega)} \lesssim ||g_1-g_2||_{H^m(\Omega)}$, we have \begin{align*} T_2 &\lesssim ||g_1-g_2||_{H^m(\Omega)} || \prod_{\sigma \in \beta} \partial^\sigma g_2||_{L^2(\Omega)} \\ 
	&\leq ||g_1-g_2||_{H^m(\Omega)} \prod_{\sigma \in \beta} ||\partial^\sigma g_2||_{L^{2|\alpha|/|\sigma|}(\Omega)} \\ 
	&\leq ||g_1-g_2||_{H^m(\Omega)} \prod_{\sigma \in \beta} ||g_2||_{H^m(\Omega)} \\
	&\leq ||g_1-g_2||_{H^m(\Omega)} (||g_1||_{H^m(\Omega)} + ||g_2-g_1||_{H^m(\Omega)})^{|\beta|} \\
	&\lesssim %q C\{K_m, ||g_1||_{\Hm{m}} \} ( 
	||g_1||_{H^m(\Omega)}^{|\beta|} ||g_1 - g_2||_{H^m(\Omega)} + \mathcal{O}(||g_1 - g_2||_{H^m(\Omega)}^2). 
	\end{align*}

	To manage $T_1$ we need Lemma \ref{lem:expand}, assuming an ordering on the multi-indices with $L(\sigma) = \{\psi \in \beta | \psi \prec \sigma \}$ and $G(\sigma) = \{\psi \in \beta | \sigma \prec \psi \}$ ($\prec$ denotes the ordering of the multi-indices in $\beta$), and employ similar logic as for $T_2$:

	\begin{align*}
	T_1 &= \left| \left| f^{(|\beta|)} (g_1) \left( \sum_{\sigma \in \beta} (\partial^\sigma g_1 - \partial^\sigma g_2) \prod_{\sigma' \in L(\sigma)} \partial^{\sigma'} g_1 \prod_{\sigma'' \in G(\sigma)} \partial^{\sigma''} g_2 \right) \right| \right|_{L^2(\Omega)} \\ 
	&\leq K_{|\beta|} \sum_{\sigma \in \beta} \left| \left| (\partial^\sigma g_1 - \partial^\sigma g_2) \prod_{\sigma' \in L(\sigma)} \partial^{\sigma'} g_1 \prod_{\sigma'' \in G(\sigma)} \partial^{\sigma''} g_2 \right| \right|_{L^2(\Omega)} \\ 
	&\leq K_{|\beta|} \sum_{\sigma \in \beta} || \partial^\sigma(g_1 - g_2)||_{L^{2|\alpha|/|\sigma|}} \prod_{\sigma' \in L(\sigma)} ||\partial^{\sigma'} g_1||_{L^{2|\alpha|/|\sigma'|}} \prod_{\sigma'' \in G(\sigma)} ||\partial^{\sigma''} g_2 ||_{L^{2|\alpha|/|\sigma''|}}\\
	&\leq K_{|\beta|} \sum_{\sigma \in \beta} ||g_1 - g_2||_{H^m(\Omega)} \max(||g_1||_{H^m(\Omega)}, ||g_2||_{H^m(\Omega)})^{|\beta|-1} \\
	&\leq K_{|\beta|} \sum_{\sigma \in \beta} ||g_1 - g_2||_{H^m(\Omega)} (||g_1||_{H^m(\Omega)} + ||g_2-g_1||_{H^m(\Omega)})^{|\beta|-1} \\
	&\lesssim ||g_1-g_2||_{H^m(\Omega)} ||g_1||_{H^m(\Omega)}^{|\beta|-1} + \mathcal{O}(||g_1-g_2||_{H^m(\Omega)}^2).
	\end{align*}	
	Here, we again used that $\frac{1}{2|\alpha|/\sigma} + \sum_{\sigma' \in L(\sigma)} \frac{1}{2|\alpha|/\sigma'} + \sum_{\sigma'' \in G(\sigma)} \frac{1}{2|\alpha|/\sigma''} = 1/2$ so that the overall product is in $L^2(\Omega)$ and by Lemma \ref{lem:sobin}, we can multiply the terms. We again remark that this bound has an $\mathcal{O}(||g_1 - g_2||_{H^m(\Omega)})$ dependence.
	
	The result of analyzing $T_1$ and $T_2$ and the $\alpha=0$ case is that $$|| f \circ g_1 - f \circ g_2 ||_{H^m(\Omega)}^2 \lesssim (1 + ||g_1||^{2m}) ||g_1 - g_2||_{H^m(\Omega)}^2 + \mathcal{O}(||g_1 - g_2||_{\Hm{m}}^3).$$	
\end{proof}

\subsection{Proofs of Major Results not Included in Main Body}

\subsubsection{Proof of Lemma \ref{lem:compm}}

\begin{proof}[Proof of Lemma \ref{lem:compm}]
We compute as follows:
\begin{align*}
||f \circ g||_{H^m(\Omega)}^2 &= \int_\Omega f(g(x))^2 \dd x + \sum_{0<|\alpha|\leq m} \int_\Omega |\partial^\alpha (f \circ g)|^2 \dd x \\
&\leq K^2 \int_\Omega g(x)^2 \dd x + \sum_{0 < |\alpha| \leq m} \int_\Omega |\sum_{\beta \in \mathcal{P}(\alpha)} f^{(|\beta|)}(g) \prod_{\sigma \in \beta} \partial^\sigma g|^2 \dd x \\
&\lesssim ||g||_{H^m(\Omega)}^2 + \sum_{0 < |\alpha| \leq m} \sum_{\beta,\beta' \in \mathcal{P}(\alpha)} \int_\Omega |f^{(|\beta|)}(g) f^{(|\beta'|)}(g) \prod_{\sigma \in \beta} \partial^\sigma g \prod_{\sigma' \in \beta'} \partial^{\sigma'} g| \dd x \\
&\lesssim ||g||_{H^m(\Omega)}^2 + \sum_{0 < |\alpha| \leq m} \sum_{\beta,\beta' \in \mathcal{P}(\alpha)} \int_\Omega | \prod_{\sigma \in \beta} \partial^\sigma g \prod_{\sigma' \in \beta'} \partial^{\sigma'} g | \dd x.
\end{align*}

At this point, we note that for each $\alpha$, $\sum_{\sigma \in \beta} |\sigma| = |\alpha| \leq m$ with $1 \leq |\beta| \leq |\alpha|$ and likewise for the primes. Using Corollary \ref{cor:okprod}: for each $\alpha$, the terms in the summand are $\lesssim ||g||_{H^m(\Omega)}^{2} + ||g||_{H^m(\Omega)}^{2 |\alpha|}$ (there are at least $1$ and no more than $|\alpha|$ factors in each $\prod$-product). Whence, summing over $\alpha$, and taking square roots, we obtain the first desideratum.

From here on, we assume that $m>d+2$ and we define $\sigma^*(m) = m/2+1$. By Sobolev embedding, all derivative orders $j$ with $j \leq \sigma^*(m)$ are in $L^\infty(\Omega).$

We also define $L^*(\alpha) = \{ \beta \in \mathcal{P}(\alpha) | \quad \max_{\sigma \in \beta} |\sigma| \leq \sigma^*(m) \}$, and  \\
$U^*(\alpha) = \{ \beta \in \mathcal{P}(\alpha) | \quad \max_{\sigma \in \beta} |\sigma| > \sigma^*(m) \}.$ This separates a partition into sets of multiindices where the maximum order is ``small" or ``not small." Note then that $\mathcal{P}(\alpha) = L^*(\alpha) \cup U^*(\alpha)$ and $|U^*(\alpha)| \leq 1.$ To see this, note that $\sum_{\sigma \in U^*(\alpha)} |\sigma| \geq |U^*(\alpha)| (m/2+1)$. Then if $|U^*(\alpha)| \geq 2$ then $\sum_{\sigma \in U^*(\alpha)} |\sigma| \geq m+2 > m$ and this is a contradiction as the sum of derivative orders cannot exceed $m$. Continuing from before,
\begin{align*}
||f \circ g||_{H^m(\Omega)}^2 &\lesssim ||g||_{H^m(\Omega)}^2 + \sum_{0 < |\alpha| \leq m} \sum_{\beta,\beta' \in \mathcal{P}(\alpha)} \int_\Omega | \prod_{\sigma \in \beta} \partial^\sigma g \prod_{\sigma' \in \beta'} \partial^{\sigma'} g | \dd x \\
&\lesssim ||g||_{H^m(\Omega)}^2 + \overbrace{\sum_{0 < |\alpha| \leq m} \sum_{\substack{\beta \in L^*(\alpha) \\ \beta' \in L^*(\alpha)}} \prod_{\sigma \in \beta} ||\partial^\sigma g||_{L^\infty(\Omega)} \prod_{\sigma' \in \beta'} ||\partial^{\sigma'} g||_{L^\infty(\Omega)}}^{T_1} \\
&+ 2 \overbrace{\sum_{0 < |\alpha| \leq m} \sum_{\substack{\beta \in L^*(\alpha) \\ \beta' \in U^*(\alpha)}} \prod_{\sigma \in \beta} ||\partial^\sigma g||_{L^\infty(\Omega)} \prod_{\substack{ \sigma' \in \beta' \\ |\sigma'| \leq \sigma^*(m)}} ||\partial^{\sigma'} g||_{L^\infty(\Omega)} \int_\Omega |\partial^{\bar \beta'} g| \dd x}^{T_2} \\
&+ \overbrace{\sum_{0 < |\alpha| \leq m} \sum_{\substack{\beta \in U^*(\alpha) \\ \beta' \in U^*(\alpha)}} \prod_{\substack{\sigma \in \beta \\ |\sigma| \leq \sigma^*(m)}} ||\partial^\sigma g||_{L^\infty(\Omega)} \prod_{\substack{ \sigma' \in \beta' \\ |\sigma'| \leq \sigma^*(m)}} ||\partial^{\sigma'} g||_{L^\infty(\Omega)} \int_\Omega |\partial^{\bar \beta} g \partial^{\bar \beta'}g| \dd x}^{T_3},
\end{align*}
where $\bar \beta$ and $\bar \beta'$ denote the single derivative of maximum order in those sets $\beta$ and $\beta'$. By construction, the $L^\infty(\Omega)$-norms of the derivatives in the $L^*$ sets are controlled by $||g||_{H^m(\Omega)}.$ We denote $\ul{C} = ||g||_{\mathcal{C}^{b^*}}$ where $b^*$ is the largest integer not exceeding $\sigma^*$. Then $$T_1, T_2, T_3 \lesssim ||g||_{H^m(\Omega)}^2 (1 + \ul{C}^{2m-2}),$$ which we can see by bounding each term of $T_1$-$T_3$ by a factor with $||g||_{H^m(\Omega)}^2$ times a remaining factor that has zero factors at minimum (the $1$) or $2m-2$ factors at maximum (the $\ul{C}^{2m-2}$). We also used that the integrand in $T_2$ is in $L^2(\Omega)$ and both factors in the integrand of $T_3$ are in $L^2(\Omega).$ We therefore have that
\begin{align*}
||f \circ g||_{H^m(\Omega)}^2 &\lesssim ||g||_{H^m(\Omega)}^2 (1 + \ul{B}^2) \implies \\
||f \circ g||_{H^m(\Omega)} &\lesssim ||g||_{H^m(\Omega)} (1 +  \ul{B})
\end{align*}
where and $\ul{B}^2 \sim \ul{C}^{2m-2}$. 

\end{proof}

\subsubsection{Proof of Lemma \ref{lem:supportenergy}}

\begin{proof}[Proof of Lemma \ref{lem:supportenergy}]
We first note that with $m>d+3$, then derivative orders up to and including $m/2+3/2$ are controlled by $||g||_{\Hm{m}}$ (in particular, derivatives of orders up to and including $m/2+1$ are controlled) and thereby $\bar B^{1/(m+1)} \lesssim ||g||_{\Hm{m}}.$ Above $m/2+3/2$, derivatives are in $\Lp{2}$ up to order $m$. The set $\sigma$ can have no more than $1$ element of order larger than $\sigma^* = m/2+1$.

If $|\varkappa|=0$ then $\sum_{\zeta \in \sigma} |\zeta| \leq |\alpha| + 2$ and $\max_{\zeta \in \sigma} |\zeta| \leq |\alpha|.$ We seek to bound $$|\int_{\Td} Q \prod_{\zeta \in \sigma} \partial^\zeta g \dd x|.$$ If there is one $\zeta \in \sigma$ with $|\zeta| > \sigma^*$, it can be pulled out into an $L^2$-product with $\partial^\alpha$ with the remaining $0$ to $|\alpha|+1$ $g$-factors bounded in $\Lp{\infty}$ by $\bar B^{1/(m+1)}$ along with $Q$ being bounded itself. If there are no such $\zeta$, any one can be chosen for an $L^2$-product with $\partial^\alpha$ and the remaining factors again bounded. The resulting bound would be $\lesssim ||g||_{\Hm{m}} (1 + \bar B^{(|\alpha|+1)/(m+1)}).$

If $|\varkappa|=1$ then $\sum_{\zeta \in \sigma} |\zeta| \leq |\alpha| + 1$ and $\max_{\zeta \in \sigma} |\zeta| \leq |\alpha|.$ Suppose $\varkappa = \{ e \}.$ We seek to bound $$|\int_{\Td} Q \partial^e g \prod_{\zeta \in \sigma} \partial^\zeta g \dd x|.$$ If there is a single $\zeta \in \sigma$ where $|\zeta| > \sigma^*$, it can be pulled out with an $L^2$-product with $\partial^\alpha$ and the remaining $1$ to $|\alpha|$ $g$-factors along with $Q$ bounded. If there are no such $\zeta$ then $\partial^\alpha g \partial^e g$ can form an $L^2$-product with the remaining $1$ to $|\alpha|+1$ $g$-factors and $Q$ bounded. This results in a bound $\lesssim ||g||_{\Hm{m}}^2 ( \bar B^{1/(m+1)} + \bar B^{(|\alpha|+1)/(m+1)}).$

Combining both results, we bound with $(1+ \bar B) ||g||_{\Hm{m}}^2$.
\end{proof}

\subsubsection{Proof of Lemma \ref{lem:sobbounds}}

\begin{proof}[Proof of Lemma \ref{lem:sobbounds}]
	1. We have \begin{align*} 
	||\rho u||_{\Hm{m}} &= ||u^+ \rho - \kappa M(\rho)||_{\Hm{m}} \\
	&\lesssim u^+ ||\rho||_{\Hm{m}} + ||\kappa||_{\Hm{m}} ||M(\rho)||_{\Hm{m}} \\
	&\lesssim_m % C\{u^-, u^+, \Xi_{\kappa,m}, \Omega, m\}
	 ||\rho||_{\Hm{m}}+||\rho||_{\Hm{m}}^{m}.
	\end{align*}
	We used Lemma \ref{lem:compm}. 
	
	2. $||\lap(\rho u)||_{\Hm{m}} \leq ||\rho u||_{\Hm{m+2}}.$
	
	3. Trivial: $||\omega \rho||_{\Hm{m}} \lesssim 
	||\omega||_{\Hm{m}} ||\rho||_{\Hm{m}}.$
	
	4. We compute
	\begin{align*}
	||I[\gamma \rho u]||_{\Hm{m}}^2 &= \sum_{|\alpha| \leq m} \int_{\dom} \left| \partial_x^\alpha (\int_{\dom} \tau(y,x,t) \gamma \rho u|_y \dd y) \right|^2 \dd x \\
	&= \sum_{|\alpha| \leq m} \int_{{\dom}} \left| \int_{\dom} \partial^\alpha_x \tau(y,x,t) \gamma \rho u|_y \dd y \right|^2 \dd x \\
	&\leq \Xi_{m}^2 \int_{\dom} \left| \int_{\dom} \gamma \rho u |_y \dd y \right|^2 \dd x \\
	&\lesssim_m \int_{\dom} ||\rho||_{\Lp{2}}^2 \dd x. 
	\end{align*}
	We used that $u \in (0,1)$ and that $\gamma$ is bounded.
	
	5. This is also trivial: $||\gamma \rho u||_{\Hm{m}} \lesssim ||\gamma||_{\Hm{m}} ||\rho u||_{\Hm{m}}.$
\end{proof}

\subsubsection{Proof of Lemma \ref{lem:lipbound}}

\begin{proof}[Proof of Lemma \ref{lem:lipbound}]
	Most proofs are trivial. Item 3 is obvious from Lemma \ref{lem:sobbounds}. %\ref{lem:moll}.
 Item 1 follows directly from proposition \ref{prop:chain} taking into account Remark \ref{rmk:fog}. Items 2, 4, and 5 follow from it. For example:	
	\begin{align*}
	||\rho_1 u_1 - \rho_2 u_2||_{\Hm{m}} &= ||u^+ (\rho_1 - \rho_2) - \kappa (M(\rho_1) - M(\rho_2)||_{\Hm{m}} \\
	&\leq u^+ ||\rho_1 - \rho_2||_{\Hm{m}} + ||\kappa||_{\Hm{m}} ||M(\rho_1) - M(\rho_2)||_{\Hm{m}}.
	\end{align*}	
\end{proof}

\subsubsection{Local Existence and Uniqueness for Regularized Nonautonomous Problem}

\begin{theorem}[Local Existence and Uniqueness for Regularized Nonautonomous Problem]	
	\label{thm:localregaut}
	For any $\epsilon > 0$ and initial condition $\rho_0(x) \in \Hm{m}$ with integer $m>d/2$ and $\rho_0 \geq \mathrm{ess} \inf_{\dom} \rho_0 > 0$ a.e., there exists a unique positive solution $\rheps \in \mathcal{C}^1 ( [0, T_{\epsilon,m}]; \Hm{m} )$  for some $T_{\epsilon,m} > 0$ to the regularized version of \sys:
	\begin{equation}
	\rho^{(\epsilon)}_t = \delta \Jeps[ \lap \left( u^+ \Jeps [\rheps] - \kappa M(\Jeps [\rheps]) \right) ] + \eta - \omega \rheps + I[\gamma \rheps \ueps] - \gamma \rheps \ueps, \label{eq:moll_pdena}
	\end{equation} 	
where $u^{(\eps)} := u(\kappa, \rheps).$
\end{theorem}

\begin{proof}[Proof of Theorem \ref{thm:localregaut}]
	We need to verify the conditions of Lemma \ref{thm:picard}. With $\Aopnon$ defined as in proposition \ref{prop:bounda} we have $$||\Aopnon[\rho_1] - \Aopnon[\rho_2]||_{\Hm{m}} \leq K||\rho_1 - \rho_2||_{\Hm{m}} + \frac{1}{\epsilon^2} \mathcal{O}(||\rho_1 - \rho_2||_{\Hm{m}}^2)$$ for a Lipschitz constant $K$ that is independent of $t$. 	

Define $$\underline{\rho}_0 = \text{ess} \inf_{\Td} \rho_0$$ and choose $\beta_1(\eps) > 0$ so that when $||\rho-\rho_0||_{\Hm{m}} \leq \beta_1$, we have $$||\rho - \rho_0||_{\Lp{\infty}} \lesssim ||\rho-\rho_0||_{\Hm{m}} < \ul{\rho}_0/2$$ and 
$$||\Jeps [\rho] - \Jeps [\rho_0]||_{\Lp{\infty}} \leq ||\rho - \rho_0||_{\Lp{\infty}} \lesssim ||\rho - \rho_0||_{\Hm{m}} < \ul{\rho}_0/2.$$

Now we choose $0<\beta(\eps) \leq \beta_1(\eps)$ so that when $\rho_1, \rho_2 \in \{ \rho | \quad ||\rho - \rho_0||_{\Hm{m}} < \beta_1 \}$, \\
$||\rho_1 - \rho_2||_{\Hm{m}} \leq \beta$ implies $||\Aopnon[\rho_1] - \Aopnon[\rho_2]||_{\Hm{m}} \leq 2K ||\rho_1 - \rho_2||_{\Hm{m}}$. This gives a uniform Lipschitz condition.
	
	We have that if $||\rho-\rho_0||_{\Hm{m}} \leq \beta$ then $\rho, \Jeps \rho > 0$ a.e. (positivity stems from the maximum principle observation given in the proof of Lemma \ref{lem:moll}). Then, by Remark \ref{rmk:Abd}, $\Aopnon$ is bounded for $\rho$ in this bounded set by a constant we call $M$.
If we choose $\alpha = \beta/M$, then on the rectangle $\mathcal{R}_0 = \{ (t, \rho) | |t| \leq \alpha, ||\rho-\rho_0||_{\Hm{m}} \leq \beta \},$ we have that $\Aopnon$ is uniformly bounded and uniformly Lipschitz in $t$. We therefore have a unique, strongly continuously differentiable solution $\rho^{(\epsilon)}$ to the regularized system valid for some time interval $0 \leq t \leq T_{\epsilon,m}.$	
\end{proof}

\section{Mollifiers on $\Td$}

\label{app:greer}

We present here proofs of various mollifier properties on $\Td$, originally presented in Greer's thesis \cite{greerThesis}. Some details have been added and some corrections have been made.

\begin{proof}[Proof of Lemma \ref{lem:moll}]
$ $ 

1. We note that $\Jeps[f]$ is the solution to the heat equation
\begin{align*}
u_t &= \lap u \\
u(x,0) &= f(x)
\end{align*}
on $\Td$ at $t=\eps.$ The inequality follows from the maximum principle. \\
To prove the uniform convergence, let $\nu>0$. We have that the Fourier series of $f$ is absolutely convergent \cite{alimov1976convergence} so $K>0$ can be chosen so that $$\sum_{k\in \mathbb{Z}^d, |k|\geq K} |\hat f(k)| \leq \nu/2.$$
%{\note The uniform convergence needs a better citation/justification.}
We have \begin{align*}
|\Jeps[f] - f| &= \lvert \sum_{k \in \mathbb{Z}^d} \hat f(k) \e^{2 \pi i k \cdot x}(1 - \e^{-\eps^2 |k|^2}) \rvert \leq \sum_{k \in \mathbb{Z}^d} |\hat f(k)| (1 - \e^{-\eps^2 |k|^2}) \\
&\leq \sum_{k \in \mathbb{Z}^d, |k| \leq K} |\hat f(k)| (1 - \e^{-\eps^2 |k|^2}) + \frac{\nu}{2} \leq (1 - \e^{-\eps^2 K^2}) \sum_{k \in \mathbb{Z}^d, |k| \leq K} |\hat f(k)| + \frac{\nu}{2} \leq \nu
\end{align*}
for small enough $\epsilon.$ \\
3. By Parseval's formula, 
\begin{align*}
(\Jeps f, g)_{\Lp{2}} &= ( \hat{\Jeps f}, \hat g)_{\ell^2(\mathbb{Z}^d)} = \sum_{k \in \mathbb{Z}^d} \hat f(k) \e^{-\epsilon^2 |k|^2} \hat g(k) \\
&= \sum_{k \in \mathbb{Z}^d} \hat f(k) \hat g(k) \e^{-\eps^2 |k|^2} = ( \hat f, \hat {\Jeps g} )_{\ell^2(\mathbb{Z}^d)} = (f, \Jeps g)_{\Lp{2}}. 
\end{align*} \\
2. Let $\Phi$ be any test function in $\mathcal{C}^\infty(\Td).$ We have
\begin{align*}
(\partial^\alpha \Jeps f, \Phi)_{\Lp{2}} &= (-1)^{|\alpha|} (\Jeps f, \partial^\alpha \Phi)_{\Lp{2}} = (-1)^{|\alpha|} (f, \Jeps \partial^\alpha \Phi)_{\Lp{2}} \\
&= (-1)^{|\alpha|} \sum_{k \in \mathbb{Z}^d} \hat f(k) (2 \pi i k )^\alpha \hat \Phi (k) \e^{-\epsilon^2 |k|^2} = (-1)^{|\alpha|} (f, \partial^\alpha \Jeps \Phi)_{\Lp{2}} \\
&= (\partial^\alpha f, \Jeps \Phi)_{\Lp{2}} = (\Jeps \partial^\alpha f, \Phi)_{\Lp{2}}
\end{align*}
and thus distributional derivatives commute with mollifiers. \\
4. Let $\nu>0$. Since $f \in \Hm{s}$, there is $K>0$ so that 
$$\sum_{k \in \mathbb{Z}^d, |k|>K} (1 + |k|^2)^s |\hat f(k)|^2 < \nu/2.$$
Picking $\eps$ small enough, we have
\begin{align*}
||\Jeps f - f||_{\Hm{s}}^2 &\leq \sum_{k \in \mathbb{Z}^d} (1 + |k|^2)^s (1 - \e^{-\eps^2 |k|^2})^2 |\hat f(k)|^2 \\
&\leq \sum_{k \in \mathbb{Z}^d, |k| \leq K} (1 + |k|^2)^s (1 - \e^{-\eps^2 |k|^2})^2 |\hat f(k)|^2 \\ &+ \sum_{k \in \mathbb{Z}^d, |k| > K} (1 + |k|^2)^s |\hat f(k)|^2 \\
&< \frac{\nu}{2} + \frac{\nu}{2}.
\end{align*}
Thus there is convergence in $\Hm{s}.$ To observe the rate of convergence in $\Hm{s-1}$ is linear in $\eps$, we use $$1 - \e^{-\theta^2} \leq \theta^2.$$ This leads to the bounds
\begin{align*}
\lvert \frac{(1 - \e^{-\eps^2 |k|^2})^2}{1 + |k|^2} \rvert &< \eps^4 \lambda^2, &\quad |k| < \lambda \\
\lvert \frac{(1 - \e^{-\eps^2 |k|^2})^2}{1 + |k|^2} \rvert &< \frac{1}{\lambda^2}, &\quad |k| \geq \lambda 
\end{align*}
so if $\lambda = 1/\epsilon$,
\begin{align*}
||\Jeps f - f||_{\Hm{s-1}}^2 &= \sum_{k \in \mathbb{Z}^d} (1 + |k|^2)^{s-1} (1 - \e^{-\eps^2 |k|^2})^2 |\hat f(k)|^2 \\
&\leq \left( \sup_{k \in \mathbb{Z}^d} \lvert \frac{ (1 - \e^{-\eps^2 |k|^2})^2}{ 1 + |k|^2} \rvert \right) ||f||_{\Hm{s}}^2 \leq \eps^2 ||f||_{\Hm{s}}^2.
\end{align*} \\
5. We first observe that the function $q(x) = (1+x)^\nu \e^{-\eps^2 x}$ has a global maximum at $x = \nu/\eps^2-1$ with value $(\frac{\nu}{\eps^2})^\nu \e^{-(\nu-\eps^2)}.$ Thus, for $\nu>0$ and $1>\epsilon>0$, $$q(|k|^2) = (1 + |k|^2)^\nu \e^{-\eps^2 |k|^2} \leq ( \frac{\nu}{\eps^2})^\nu \e^{-\nu+\eps^2} \leq C \eps^{-2\nu},$$ for a $C$ depending on $\nu$, and $||\Jeps f||_{\Hm{m+\nu}} \lesssim \eps^{-\nu} ||f||_{\Hm{m}}.$ Next, we estimate the representation of $\Jeps \partial^\nu f(x)$ in terms of its Fourier series. Let $\tilde \nu$ be a multi-index with $|\tilde \nu| = \nu.$
\begin{align*}
|\Jeps \partial^{\tilde \nu} f(x)| &= |\sum_{k \in \mathbb{Z}^d} \e^{-\eps^2 |k|^2} (2 \pi i k)^{\tilde \nu} \e^{-2 \pi i k \cdot x} \hat f(k)| \\
&= | \sum_{k \in \mathbb{Z}^d} (\e^{-\frac{\eps^2 |k|^2}{2}}) \left( (2 \pi i k)^{\tilde \nu} \e^{2 \pi i k \cdot x} \e^{-\frac{\eps^2 |k|^2}{2}} \hat f(k) \right)| \\
&\leq \left( \sum_{k \in \mathbb{Z}^d} \e^{-\eps^2 |k|^2} \right)^{1/2} \left( \sum_{k \in \mathbb{Z}^d} |(2 \pi i k)^{2 \tilde \nu} \e^{-\eps^2 |k|^2} |\hat f(k)|^2 \right)^{1/2} \\
& \leq \left( \sum_{k \in \mathbb{Z}^d} \e^{-\eps^2 |k|^2} \right)^{1/2} ||\Jeps \partial^{\tilde \nu} f||_{\Hm{0}} \\ &\lesssim \left( \sum_{k \in \mathbb{Z}^d} \e^{-\eps^2 |k|^2} \right)^{1/2} \frac{1}{\eps^{\nu-m}} ||f||_{\Hm{m}}.
\end{align*}
Now we examine the last inequality: \\

\begin{align*}
\sum_{k \in \mathbb{Z}^d} \e^{-\eps^2 |k|^2} &= \sum_{k \in \mathbb{Z}^d} \prod_{i=1}^d \e^{-\eps^2 k_i^2} \leq 2^d \sum_{k \in \mathbb{Z}_{\geq 0}^d} \prod_{i=1}^d \e^{-\eps^2 k_i^2} \\
&= 2^d \sum_{k_1=0}^\infty  ( \e^{-\eps^2 k_1^2} ( \sum_{k_2=0}^\infty \e^{-\eps^2 k_2^2} ( ... \sum_{k_d=0}^\infty \e^{-\eps^2 k_d^2} ) ... ) ) \\
&=2^d \prod_{i=1}^d \sum_{k_i=0}^{\infty} \e^{-\eps^2 k_i^2} \\ &\leq 2^d ( \sum_{k_i=0}^\infty \e^{-\eps^2 k_i})^d = 2^d (\frac{1}{1 - \e^{-\eps^2}})^d \\ &\leq 2^d \eps^{-d}.
\end{align*} 
This gives the desired bound.
\end{proof}

\end{document}